\documentclass[11pt,reqno]{amsart}

\usepackage[T1]{fontenc}
\usepackage[utf8]{inputenc}
\usepackage{lmodern}
\usepackage{amsmath,amssymb,amsthm,mathtools}
\usepackage{mathrsfs}
\usepackage{microtype}
\usepackage{enumitem}
\usepackage[hidelinks]{hyperref}
\usepackage{comment}
\usepackage{xcolor}

\newtheorem{theorem}{Theorem}[section]
\newtheorem{proposition}[theorem]{Proposition}
\newtheorem{lemma}[theorem]{Lemma}
\newtheorem{corollary}[theorem]{Corollary}
\newtheorem{definition}[theorem]{Definition}

\theoremstyle{remark}
\newtheorem{remark}[theorem]{Remark}

\usepackage[capitalize,nameinlink]{cleveref}

\newcommand{\dd}{\,d}
\newcommand{\abs}[1]{\lvert #1\rvert}

\newcommand{\one}{\mathbf 1}

\newcommand{\Mp}{M_p}
\newcommand{\Pp}{\mathcal P_p}
\newcommand{\Wp}{W_p}

\title[Small-Order Metric-Valued Nonlocal Energies]{Decoupling and Tail Laws for Small-Order Metric-Valued Nonlocal Energies: A Structural View of the Maz'ya--Shaposhnikova Formula}
\author{Andrea Pinamonti}
\address{Dipartimento di Matematica, Università degli Studi di Trento, Via Sommarive, 14, 38123 Povo (Trento), Italia}
\email{andrea.pinamonti@unitn.it}

\begin{document}

\begin{abstract}
We develop a measure-theoretic framework for small-order limits of nonlocal
energies with metric-valued maps. For globally $L^p$ maps, we consider
interaction measures whose marginals have uniformly bounded $L^\infty$
densities $a_\alpha$ and $b_\alpha$. If these densities converge weakly-* to
$a_\infty$ and $b_\infty$, the normalized energies converge to the $L^p$
energy weighted by $a_\infty+b_\infty$ if and only if the interaction
measures escape every bounded rectangle. For maps that are only locally
$L^p$, we introduce target laws describing the distribution of the values
sampled at infinity. Convergence of their $p$-distance profiles yields the
limiting interaction energy for every $1\leq p<\infty$ and arbitrary Polish
targets, with integrated $p$-Wasserstein convergence as a sufficient
criterion. The arguments include the endpoint $p=1$ and require no linear
structure on the target space. Applications include directional and
anisotropic Maz'ya--Shaposhnikova formulas, a Bernoulli-law interpretation of
fractional perimeters, and an Abelian principle for heat-semigroup energies.
\end{abstract}
\maketitle

\section{Introduction}
For $0<s<1$, $1\leq p<\infty$, and a measurable map
$u:\mathbb R^d\to\mathbb R$, consider the fractional interaction
\[
 [u]_{W^{s,p}(\mathbb R^d)}^p
 :=
 \int_{\mathbb R^d}\int_{\mathbb R^d}
 \frac{|u(x)-u(y)|^p}{|x-y|^{d+sp}}\,\dd y\,\dd x.
\]
The two endpoint regimes have different geometric content. In the
Bourgain--Brezis--Mironescu regime $s\uparrow1$, the
interaction kernel concentrates near the diagonal and detects the local
regularity of $u$.  More precisely, for $p>1$ and
$u\in W^{1,p}(\mathbb R^d)$,
\[
 \lim_{s\uparrow1}(1-s)[u]_{W^{s,p}(\mathbb R^d)}^p
 =\kappa_{d,p}\int_{\mathbb R^d}|\nabla u|^p\,\dd x, \tag{BBM}
\]
where $\kappa_{d,p}:=\frac{1}{p}\int_{\mathbb S^{d-1}}
|\omega\cdot e|^p\,d\mathcal H^{d-1}(\omega)$, $e\in \mathbb{S}^{d-1}$ ; see \cite{BBM2001}, as well as
\cite{BourgainNguyen,BrenaPasqualettoPinamonti,BrenaPinamonti,
DavoliDiFrattaPagliari,GennaioliStefani,NguyenSobolev,Ponce2004}.
The endpoint $p=1$ and the corresponding characterization of functions of
bounded variation were established by D\'avila \cite{Davila2002}.  Related
metric-measure results were obtained in
\cite{LahtiPinamontiZhou,LahtiPinamontiZhouJGA,Oleinik2025,
MaalaouiPinamonti,SerraCassanoZambanini}.

The small-order regime $s\downarrow0$ has a different nature.  The
Maz'ya--Shaposhnikova formula \cite{MazyaShaposhnikova2002} states that,
for every $u\in\bigcup_{0<s<1}W^{s,p}(\mathbb R^d)$,
\begin{equation}
 \lim_{s\downarrow0}s[u]_{W^{s,p}(\mathbb R^d)}^p
 =\frac{2\sigma_{d-1}}p
 \int_{\mathbb R^d}|u(x)|^p\,\dd x,
 \qquad
 \sigma_{d-1}:=\mathcal H^{d-1}(\mathbb S^{d-1}). \tag{MS}
\end{equation}
After multiplication by $s$, interactions at every fixed bounded scale
vanish, while the surviving mass is carried by pairs whose distance tends to
infinity. At the level of the normalized integral, the cost
$\lvert u(x)-u(y)\rvert^p$ then decouples into one-point contributions.

Ludwig \cite{Ludwig2014} established anisotropic versions of the
Maz'ya--Shaposhnikova formula.  The corresponding asymptotics of fractional
perimeters were studied by Dipierro--Figalli--Palatucci--Valdinoci
\cite{DipierroFigalliPalatucciValdinoci2013}.  Fractional
Orlicz--Sobolev asymptotics on Carnot groups were obtained in
\cite{CapolliMaioneSalortVecchi2021}.  A geometric interpretation in terms
of volume growth at infinity was developed by Han \cite{Han2024} and, in a
metric-measure framework containing Carnot groups and spaces satisfying
Bishop--Gromov-type estimates, by Han, Xu, Zambanini and the author
\cite{HanPinamontiXuZambanini2025}.  Almost sharp conditions for broad
classes of Euclidean kernels, including non-symmetric and non-homogeneous
interactions, were recently obtained in
\cite{DavoliDiFrattaGiorgioPinamonti2025}. A recent geometric
characterization appears in \cite{GaoYuZhang2026}.\\

These results motivate a product-measure formulation in which the
small-order limit is determined by the escape of normalized interaction mass
from bounded subsets of the product space.

More precisely, let $(\Pi_\alpha)_{\alpha}$ be a family of positive measures on
$X\times X$, with marginals $a_\alpha\mu$ and $b_\alpha\mu$, where
$a_\alpha,b_\alpha\in L^\infty_+(X,\mu)$, and write
\[
\mathcal E_\alpha(u)
:=
\int_{X\times X} d_Y(u(x),u(y))^p\,\dd\Pi_\alpha(x,y).
\]
Assume that $a_\alpha\stackrel{*}{\rightharpoonup}a_\infty$ and
$b_\alpha\stackrel{*}{\rightharpoonup}b_\infty$ in $L^\infty(X,\mu)$.\\

\noindent Our first question is then:
\begin{center}
\textit{Given a metric space $(Y,d_{Y})$, which asymptotic condition on $\Pi_\alpha$ characterizes the decoupling}
\end{center}
\[
\mathcal E_\alpha(u)
\longrightarrow
\int_X(a_\infty+b_\infty)d_Y(u(x),y_0)^p\,\dd\mu(x)
\quad\textit{for every }u\in L^p(X;Y)\textit{?}
\]

The second is:

\begin{center}\textit{What replaces the limit above when the map is only locally $L^p$ and retains nontrivial states at infinity?}
\end{center}

The latter issue is closely related to the mass-at-infinity
formalism introduced by Davoli--Fanizza--Picerni
\cite{DavoliFanizzaPicerni2026} in the  Euclidean setting.  On
Riemannian manifolds, Caselli--Gennaioli \cite{CaselliGennaioli} found an
analogous phenomenon in which the small-order limit is sensitive to
stochastic completeness and, in the non-conservative case, to a bounded
harmonic function recording the mass that survives at infinity. The results described
below provide a common measure-theoretic answer to the preious two questions.\\

\textit{Decoupling.}
Let $(X,d_X)$ be a proper metric space, let $\mu$ be a positive Radon
measure finite on bounded sets, and let $(Y,d_Y,y_0)$ be a pointed metric
space containing at least two points.  Consider a net
$(\Pi_\alpha)_\alpha$ of positive Borel measures on $X\times X$ whose
marginals satisfy
\[
 (\pi_1)_\#\Pi_\alpha=a_\alpha\mu,
 \qquad
 (\pi_2)_\#\Pi_\alpha=b_\alpha\mu,
\]
where
\[
 \sup_\alpha\bigl(\|a_\alpha\|_\infty+\|b_\alpha\|_\infty\bigr)<\infty,
 \qquad
 a_\alpha\stackrel{*}{\rightharpoonup}a_\infty,
 \qquad
 b_\alpha\stackrel{*}{\rightharpoonup}b_\infty.
\]
For $u\in L^p(X;Y)$, set
\[
 J_{\Pi_\alpha}(u)
 :=\int_{X\times X}d_Y(u(x),u(y))^p\,\dd\Pi_\alpha(x,y),
 \qquad
 M_p(u):=\int_Xd_Y(u(x),y_0)^p\,\dd\mu(x).
\]
Theorem~\ref{thm:abstract-decoupling} proves the equivalence
\[
 \Pi_\alpha(K\times K)\longrightarrow0
 \quad\text{for every bounded Borel set }K\subset X
\]
if and only if
\[
 J_{\Pi_\alpha}(u)
 \longrightarrow
 \int_X(a_\infty+b_\infty)d_Y(u,y_0)^p\,\dd\mu
 \quad\text{for every }u\in L^p(X;Y).
\]
We call the first property \emph{escape from bounded rectangles}. Under the
stated marginal assumptions, it is equivalent to metric-valued decoupling.
Necessity is detected by maps taking only two values; sufficiency follows
from a compact-support decomposition and a uniform $L^p$-stability estimate.
The proof uses only the metric structure of the target and includes $p=1$.\\

\textit{Target laws and energy profiles at infinity.}
Let
$\Omega\subset X$ be bounded, let $K_\varepsilon(x,\cdot)$ be a finite
positive kernel (see Section \ref{sec:tail-laws} for the definition) from $\Omega$ to $X$, and write
\[
 m_\varepsilon(x):=K_\varepsilon(x,X).
\]
For $m_\varepsilon(x)>0$, define the probability law on $Y$ seen from $x$ by
\[
 \lambda_{\varepsilon,x}
 :=u_\#\left(\frac{K_\varepsilon(x,\cdot)}{m_\varepsilon(x)}\right).
\]
Now suppose that $Y$ is Polish, $u|_\Omega\in L^p(\Omega;Y)$, and that, for
some $\lambda_\infty\in\mathcal P_p(Y)$, the associated $p$-distance
profiles converge in the sense that, for every $a\in Y$,
\[
 \left(\int_Yd_Y(a,z)^p\,\dd\lambda_{\varepsilon,x}(z)\right)^{1/p}
 \longrightarrow
 \left(\int_Yd_Y(a,z)^p\,\dd\lambda_\infty(z)\right)^{1/p}
 \quad\text{in }L^p(\Omega).
\]
If, moreover, $m_\varepsilon\to m$ in $L^\infty(\Omega)$, then
Theorem~\ref{thm:metric-tail-law} gives
 \begin{equation*}
 \int_\Omega\int_X
 d_Y(u(x),u(y))^pK_\varepsilon(x,\dd y)\,\dd\mu(x) \longrightarrow
 \int_\Omega m(x)\int_Yd_Y(u(x),z)^p
 \,\dd\lambda_\infty(z)\,\dd\mu(x).
 \end{equation*}
Thus the state at infinity is represented by a probability law rather than
by a scalar moment. The quadratic case in
\cite{DavoliFanizzaPicerni2026} is recovered from the first two moments; the
profile formulation works for every
$1\leq p<\infty$ and avoids algebraic expansions of the interaction cost.

A first application is a directional Maz'ya--Shaposhnikova formula. Let
\(\Omega\subset\mathbb R^d\) be bounded and let
\(u\in L^p_{\mathrm{loc}}(\mathbb R^d;Y)\). Suppose that there exists a Borel map
\[
u_\infty:\mathbb S^{d-1}\to Y,
\qquad
d_Y(u_\infty(\cdot),y_0)\in L^p(\mathbb S^{d-1}),
\]
such that,
\[
 \lim_{s\downarrow0}s\int_{B_1^c}
 \frac{d_Y(u(y),a)^p}{|y|^{d+sp}}\,\dd y
 =
 \frac1p\int_{\mathbb S^{d-1}}
 d_Y(a,u_\infty(\theta))^p\,\dd\mathcal H^{d-1}(\theta)
 \qquad\forall a\in Y.
\]
Set
\[
Q_\Omega:=
(\mathbb R^d\times\mathbb R^d)\setminus(\Omega^c\times\Omega^c)
\]
and define the localized fractional interaction by
\[
 \mathcal E_{s,p}^{\Omega}(u)
 :=
 \iint_{Q_\Omega}
 \frac{d_Y(u(x),u(y))^p}{|x-y|^{d+sp}}
 \,\dd y\,\dd x.
\]
If \(\mathcal E_{s_0,p}^{\Omega}(u)<\infty\) for some \(s_0>0\), then
Theorem~\ref{thm:directional-metric-MS} yields
\[
 \lim_{s\downarrow0}\frac{s}{2}\mathcal E_{s,p}^{\Omega}(u)
 =
 \frac1p\int_\Omega\int_{\mathbb S^{d-1}}
 d_Y(u(x),u_\infty(\theta))^p
 \,\dd\mathcal H^{d-1}(\theta)\,\dd x.
\]
For $Y=\mathbb R$ and $p=2$, this condition reduces to convergence of the
first two moments at infinity and the formula agrees with
\cite[Theorem~2.11]{DavoliFanizzaPicerni2026}. The argument also applies to norms induced by
origin-symmetric convex bodies.  This gives the metric-valued anisotropic
formula of Proposition~\ref{prop:anisotropic-directional-MS} and recovers
Ludwig's theorem \cite{Ludwig2014} when the state at infinity vanishes and
the map is compactly supported.\\

\textit{Singular energies and fractional perimeters.} If
$E\subset\mathbb R^n$ has finite fractional perimeter in $\Omega$ at one
positive order and
\[
 \alpha(E):=\lim_{s\downarrow0}
 s\int_{E\cap B_1^c}\frac{\dd y}{|y|^{n+s}}
\]
exists, set
\[
 \vartheta_E:=\frac{\alpha(E)}{\sigma_{n-1}},
 \qquad
 \lambda_\infty^E:=(1-\vartheta_E)\delta_0+\vartheta_E\delta_1.
\]
Proposition~\ref{prop:Bernoulli-tail-perimeter} then yields
\[
 \lim_{s\downarrow0}s\,\operatorname{Per}_s(E;\Omega)
 =\sigma_{n-1}
 \bigl[(1-\vartheta_E)|E\cap\Omega|
       +\vartheta_E|E^c\cap\Omega|\bigr].
\]
This recovers the formula of
Dipierro--Figalli--Palatucci--Valdinoci
\cite{DipierroFigalliPalatucciValdinoci2013} and identifies its
mass-at-infinity parameter as the parameter of a Bernoulli target law.\\

\textit{Heat-semigroup energies.}
Let $q_t(x,y)$ have marginal densities
$a_t,b_t\in L^\infty_+(X,\mu)$, and set
\[
 J_p(t,u):=\int_X\int_Xd_Y(u(x),u(y))^pq_t(x,y)
 \,\dd\mu(y)\,\dd\mu(x).
\]
If the corresponding couplings escape from bounded rectangles as
$t\to\infty$ and $a_t,b_t$ converge weakly-* to
$a_\infty,b_\infty$, then
\[
 J_p(t,u)\longrightarrow
 \int_X(a_\infty+b_\infty)d_Y(u,y_0)^p\,\dd\mu.
\]
An Abelian argument transfers this large-time limit to the small-order
asymptotics of
\[
 \mathscr N_{s,p}^{q,\beta}(u)^p
 :=\int_0^\infty t^{-\beta sp-1}J_p(t,u)\,\dd t.
\]
Under the short-time integrability assumption in
Theorem~\ref{thm:Abel-semigroup}, one obtains
\[
 \lim_{s\downarrow0}s\mathscr N_{s,p}^{q,\beta}(u)^p
 =\frac1{\beta p}\int_X(a_\infty+b_\infty)
 d_Y(u,y_0)^p\,\dd\mu.
\]

This formula determines the constants in the
non-symmetric hypoelliptic setting of Buseghin--Garofalo--Tralli
\cite{BuseghinGarofaloTralli2022}.

Moreover, on a complete connected Riemannian manifold of infinite volume, the
marginals need not be constant multiples of the volume measure.  If $h_t$
is the minimal heat kernel and
\[
 \theta_M(x):=\lim_{t\to\infty}\int_Mh_t(x,y)\,\dd\mu(y),
\]
Corollary~\ref{cor:riemannian-infinite-volume} gives
\[
 \lim_{s\downarrow0}s\mathscr N_{s,p}^{h,1/2}(u)^p
 =\frac4p\int_M\theta_M(x)d_Y(u(x),y_0)^p\,\dd\mu(x).
\]
This extends the large-time result of
Caselli--Gennaioli \cite{CaselliGennaioli} from scalar-valued maps and
$p=2$ to arbitrary metric targets and every $1\leq p<\infty$.  Further
applications include infinite weighted graphs and the dimension-free Baouendi--Grushin formula of Li--Wang
\cite{LiWang2024}.\\

\textit{Carnot groups.}
Finally, let $G$ be a Carnot group of homogeneous dimension $Q$, equipped
with a homogeneous norm $N_G$, and let $\sigma_G$ be the
corresponding polar measure on
\[
 S_G:=\{h\in G:N_G(h)=1\}.
\]
For $u\in L^p(G;Y)$, define
\[
 E_{s,p}(u):=\int_G\int_G
 \frac{d_Y(u(x),u(y))^p}{N_G(x^{-1}y)^{Q+sp}}\,\dd y\,\dd x.
\]
If $E_{\theta,p}(u)<\infty$ for some $\theta>0$,
Theorem~\ref{thm:Carnot-MS} gives
\[
 \lim_{s\downarrow0}sE_{s,p}(u)
 =\frac{2\sigma_G(S_G)}p\int_Gd_Y(u(x),y_0)^p\,\dd x.
\]
No differential structure of the target is used. In particular, the
result applies to maps between Carnot groups endowed with left-invariant
homogeneous distances. The scalar Carnot-group formula is recovered by
taking $Y=\mathbb R$. Similarly, the scalar
heat-semigroup formula of Garofalo--Tralli
\cite[Theorem~1.5]{GarofaloTralli2024} is recovered when two Haar marginals
of unit mass survive at large time.\\

The approach developed in the present paper can be generalized to the case $s\uparrow 1$, the results will appear in \cite{Pin}.\\

\textbf{Organization of the paper.}
Section~\ref{metric} establishes the abstract two-marginal decoupling
theorem and characterizes metric-valued decoupling by escape from
bounded rectangles.  Section~\ref{sec:tail-laws} develops the theory of
target laws and their energy profiles, and derives the Euclidean directional
and anisotropic formulas.
Section~\ref{abstract} separates short- and long-range interactions and
interprets fractional-perimeter asymptotics through Bernoulli laws at
infinity.  Section~\ref{sec:semigroups} proves the Abelian semigroup
principle and applies it to non-symmetric hypoelliptic semigroups,
Riemannian manifolds, weighted graphs, and Baouendi--Grushin structures.
Section~\ref{sec:Carnot} concludes with the metric-valued
Maz'ya--Shaposhnikova formula on Carnot groups.\\

\textbf{Acknowledgments:} The author thanks Elisa Davoli for some comments on a preliminary version of the paper.

\section{Metric-valued \texorpdfstring{$L^p$}{Lp} maps, two-marginal couplings and the abstract decoupling theorem}\label{metric}

Throughout this section $(X,d_X)$ denotes a proper metric
space\footnote{A metric space $(X,d_X)$ is called proper if every closed
bounded subset of $X$ is compact.} and $\mu$ a positive Radon measure, finite on bounded sets and not necessarily globally finite. We also fix a pointed metric space $(Y,d_Y,y_0)$ and 
$1\le p<\infty$.

\begin{definition}
A measurable map $u:X\to Y$ belongs to $L^p(X;Y)$, relative to the chosen
base point $y_0$, if
\begin{equation*}
 \Mp(u):=\int_X d_Y(u(x),y_0)^p\,\dd\mu(x)<\infty.
\end{equation*}
For $u,v\in L^p(X;Y)$ we set
\begin{equation*}
 d_p(u,v):=
 \left(\int_X d_Y(u(x),v(x))^p\,\dd\mu(x)\right)^{1/p}.
\end{equation*}
\end{definition}

The dependence on the base point is harmless on a finite-measure source,
but it matters on an infinite-measure space.  We therefore keep the pointed
target throughout the paper.

\begin{lemma}
\label{lem:truncation}
For every $u\in L^p(X;Y)$ there exists a sequence $u_j\in L^p(X;Y)$ such
that
\begin{enumerate}[label=\textup{(\roman*)}]
\item $u_j=y_0$ outside a compact set $K_j\subset X$;
\item $d_Y(u_j(x),y_0)\le j$ for every $x\in X$;
\item $d_p(u_j,u)\to0$.
\end{enumerate}
\end{lemma}

\begin{proof}
  The finite Borel measure
$d_Y(u(x),y_0)^p\mu$ is Radon.  Hence for every $j\in\mathbb{N}$ one can choose a compact
$K_j\subset X$ such that
\[
 \int_{X\setminus K_j}d_Y(u(x),y_0)^p\,\dd\mu<2^{-j}.
\]
Define
\[
 u_j(x):=
 \begin{cases}
  u(x),&x\in K_j\text{ and }d_Y(u(x),y_0)\le j,\\
  y_0,&\text{otherwise}.
 \end{cases}
\]
Then (i) and (ii) are immediate.  Moreover,
\begin{align*}
 d_p(u_j,u)^p
 &=\int_{\{u_j\ne u\}}d_Y(u(x),y_0)^p\,\dd\mu(x)\\
 &\le
 \int_{X\setminus K_j}d_Y(u(x),y_0)^p\,\dd\mu(x)
 +\int_{\{d_Y(u(x),y_0)>j\}}d_Y(u(x),y_0)^p\,\dd\mu(x).
\end{align*}
The first term tends to zero by construction and the second tends to zero by
dominated convergence.  Hence (iii) follows.
\end{proof}

\begin{definition}
\label{def:bounded-marginal-coupling}
Let $a,b\in L^\infty_+(X,\mu)$. A positive Borel measure $\Pi$ on
$X\times X$ is called an \emph{$(a,b)$-marginal coupling} of $\mu$ if
\begin{equation*}
 (\pi_1)_\#\Pi=a\mu,
 \qquad
 (\pi_2)_\#\Pi=b\mu,
\end{equation*}
where $\pi_1(x,y)=x$ and $\pi_2(x,y)=y$. Equivalently, for every
nonnegative Borel function $\varphi$,
\[
 \int_{X\times X}\varphi(x)\,\dd\Pi(x,y)
 =\int_X a\varphi\,\dd\mu,
 \qquad
 \int_{X\times X}\varphi(y)\,\dd\Pi(x,y)
 =\int_X b\varphi\,\dd\mu.
\]
\end{definition}

For such a coupling and $u\in L^p(X;Y)$, define the \emph{increment
energy}
\begin{equation*}
 J_\Pi(u):=
 \int_{X\times X}d_Y(u(x),u(y))^p\,\dd\Pi(x,y).
\end{equation*}
The triangle inequality and the marginal identities give
\begin{equation}\label{eq:J-Pi-bound}
 J_\Pi(u)
 \le 2^{p-1}\bigl(\|a\|_\infty+\|b\|_\infty\bigr)\Mp(u)<+\infty.
\end{equation}

\begin{definition}
\label{def:escape-rectangles}
Let $(\Pi_\alpha)_{\alpha}$ be a family of positive Borel measures on $X\times X$ indexed by a directed set\footnote{A partially ordered set $(I,\preceq)$ is called a \emph{directed set} if, for every
$i,j\in I$, there exists $k\in I$ such that
\[
i\preceq k
\qquad\text{and}\qquad
j\preceq k.
\]}. We say that $(\Pi_\alpha)_{\alpha}$ \emph{escapes from bounded rectangles} if, for every bounded Borel
set $K\subset X$,
\begin{equation}\label{eq:escape-rectangles}
 \Pi_\alpha(K\times K)\longrightarrow0,
\end{equation}
where the limit is understood along the directed set.
\end{definition}
\begin{remark}\label{compact}
Since $X$ is proper, it is equivalent to test \eqref{eq:escape-rectangles}
on compact sets. Indeed, if $K\subset X$ is bounded and Borel, then
$\overline K$ is compact and
\[
0\leq \Pi_{\alpha}(K\times K)\leq \Pi_{\alpha}(\overline{K}\times \overline{K}).
\]
Thus, if $\Pi_{\alpha}(C\times C)\to 0$ for every compact $C\subset X$,
then $\Pi_{\alpha}(K\times K)\to 0$.  The converse is immediate because
every compact subset of a metric space is bounded and Borel.
\end{remark}

\begin{lemma}
\label{lem:coupling-stability}
Let $\Pi$ be an $(a,b)$-marginal coupling of $\mu$. Then for every
$u,v\in L^p(X;Y)$,
\begin{equation*}
 \abs{J_\Pi(u)^{1/p}-J_\Pi(v)^{1/p}}
 \le \bigl(\|a\|_\infty^{1/p}+\|b\|_\infty^{1/p}\bigr)d_p(u,v).
\end{equation*}
\end{lemma}

\begin{proof}
The reverse triangle inequality in 
$L^p(X\times X,\Pi)$ gives
\begin{align*}
 &\abs{J_\Pi(u)^{1/p}-J_\Pi(v)^{1/p}}\\
 &\qquad\le
 \left(
 \int_{X\times X}
 \abs{d_Y(u(x),u(y))-d_Y(v(x),v(y))}^p\,\dd\Pi
 \right)^{1/p}.
\end{align*}
By the metric triangle inequality,
\[
 \abs{d_Y(u(x),u(y))-d_Y(v(x),v(y))}
 \le d_Y(u(x),v(x))+d_Y(u(y),v(y)).
\]
Minkowski's inequality and the marginal identities imply
\begin{align*}
 &\left(
 \int_{X\times X}
 \abs{d_Y(u(x),u(y))-d_Y(v(x),v(y))}^p\,\dd\Pi
 \right)^{1/p}\\
 &\quad\le
 \left(\int d_Y(u(x),v(x))^p\,\dd\Pi\right)^{1/p}
 +\left(\int d_Y(u(y),v(y))^p\,\dd\Pi\right)^{1/p}\\
 &\quad\le
 \|a\|_\infty^{1/p}d_p(u,v)
 +\|b\|_\infty^{1/p}d_p(u,v).
\end{align*}
This concludes the proof.
\end{proof}

\begin{proposition}
\label{prop:compact-decoupling}
Let $\Pi$ be an $(a,b)$-marginal coupling and suppose that $u:X\to Y$
satisfies $u=y_0$ outside a bounded Borel set $K$ and
$d_Y(u(x),y_0)\le L$ for every $x\in X$. Then
\begin{equation}\label{eq:compact-decomposition}
 J_\Pi(u)=\int_X(a+b)d_Y(u,y_0)^p\,\dd\mu+R_\Pi(u;K),
\end{equation}
where
\begin{equation}\label{eq:R-compact}
 R_\Pi(u;K)=
 \int_{K\times K}
 \Bigl[d_Y(u(x),u(y))^p
 -d_Y(u(x),y_0)^p-d_Y(u(y),y_0)^p\Bigr]\,\dd\Pi
\end{equation}
and
\begin{equation}\label{eq:R-bound}
 \abs{R_\Pi(u;K)}\le 2^{p+1}L^p\Pi(K\times K).
\end{equation}
\end{proposition}

\begin{proof}
 Since $u=y_0$ on $K^c$, we get
\[
 d_Y(u(x),u(y))^p-d_Y(u(x),y_0)^p-d_Y(u(y),y_0)^p=0
 \quad\text{for }(x,y)\notin K\times K.
\]
Hence
\begin{align*}
 J_\Pi(u)
 &=\int_{X\times X}
 \bigl(d_Y(u(x),y_0)^p+d_Y(u(y),y_0)^p\bigr)\,\dd\Pi\\
 &\quad+\int_{K\times K}
 \bigl[d_Y(u(x),u(y))^p-d_Y(u(x),y_0)^p-d_Y(u(y),y_0)^p\bigr] \,\dd\Pi.
\end{align*}
The first integral equals
$\int_X(a+b)d_Y(u,y_0)^p\,\dd\mu$ by the marginal identities. On
$K\times K$, the triangle inequality gives
\[
 d_Y(u(x),u(y))^p\le2^pL^p.
\]
Thus the absolute value of the bracket in \eqref{eq:R-compact} is bounded
by $2^{p+1}L^p$, which gives \eqref{eq:R-bound}.
\end{proof}

\begin{theorem}
\label{thm:abstract-decoupling}
Assume that $Y$ contains at least two points. Let
$a_\alpha,b_\alpha,a_\infty,b_\infty\in L^\infty_+(X,\mu)$, indexed by a
directed set, and suppose that
\begin{equation}\label{eq:marginal-assumptions}
 \sup_\alpha\bigl(\|a_\alpha\|_\infty+\|b_\alpha\|_\infty\bigr)<\infty,
 \qquad
 a_\alpha\stackrel{*}{\rightharpoonup}a_\infty,
 \qquad
 b_\alpha\stackrel{*}{\rightharpoonup}b_\infty
 \quad\text{in }L^\infty(X,\mu).
\end{equation}
Here weak-* convergence means convergence of the integrals against every
function in $L^1(X,\mu)$.
Let $(\Pi_\alpha)_\alpha$ be a family of
$(a_\alpha,b_\alpha)$-marginal couplings. Then the following statements are
equivalent:
\begin{enumerate}[label=\textup{(\roman*)}]
\item \label{item:escape-iff}
$\Pi_\alpha$ escapes from bounded rectangles.
\item \label{item:allmaps-iff}
For every $u\in L^p(X;Y)$,
\begin{equation}\label{eq:abstract-decoupling}
 J_{\Pi_\alpha}(u)
 \longrightarrow
 \int_X(a_\infty+b_\infty)d_Y(u,y_0)^p\,\dd\mu.
\end{equation}
\item \label{item:twovalued-iff}
Fix $y_1\in Y\setminus\{y_0\}$.  For every bounded Borel set
$K\subset X$, the map
\[
 u_K(x):=\begin{cases}y_1,&x\in K,\\y_0,&x\notin K,\end{cases}
\]
verifies \eqref{eq:abstract-decoupling}.
\end{enumerate}
\end{theorem}

\begin{proof}
We first prove \ref{item:escape-iff}$\Rightarrow$\ref{item:allmaps-iff}.
Let $u_j$ be the truncations from Lemma~\ref{lem:truncation}.  For fixed $j\in \mathbb{N}$,
$u_j$ equals $y_0$ outside a compact set $K_j$ and $d_Y(u_j(x),y_0)\leq j$ on $X$.  By
Proposition~\ref{prop:compact-decoupling},
\[
 J_{\Pi_\alpha}(u_j)
 =\int_X(a_\alpha+b_\alpha)d_Y(u_j,y_0)^p\,\dd\mu+R_{\alpha,j},
 \qquad
 |R_{\alpha,j}|\le C_{p,j}\Pi_\alpha(K_j\times K_j)\to0.
\]
Here $C_{p,j}:=2^{p+1}j^p$, and the convergence follows from
Remark~\ref{compact}.
Since $d_Y(u_j,y_0)^p\in L^1(X,\mu)$, weak-* convergence gives
\begin{equation*}
 J_{\Pi_\alpha}(u_j)
 \longrightarrow
 \int_X(a_\infty+b_\infty)d_Y(u_j,y_0)^p\,\dd\mu.
\end{equation*}
By Lemma~\ref{lem:coupling-stability} and
\eqref{eq:marginal-assumptions}, there is a constant $C>0$ independent of
$\alpha$ and $j$ such that
\begin{equation*}
 |J_{\Pi_\alpha}(u)^{1/p}-J_{\Pi_\alpha}(u_j)^{1/p}|
 \le C d_p(u,u_j).
\end{equation*}
Set
\[
 J_\infty(v):=\int_X(a_\infty+b_\infty)d_Y(v,y_0)^p\,\dd\mu.
\]
The reverse triangle inequality in
$L^p(X,(a_\infty+b_\infty)\mu;Y)$ gives
\[
 |J_\infty(u)^{1/p}-J_\infty(u_j)^{1/p}|
 \le
 \|a_\infty+b_\infty\|_\infty^{1/p}d_p(u,u_j).
\]
Now
\begin{align*}
&|J_{\Pi_\alpha}(u)^{1/p}-J_\infty(u)^{1/p}|\\
&\leq |J_{\Pi_\alpha}(u)^{1/p}-J_{\Pi_\alpha}(u_j)^{1/p}|\\
&\quad+|J_\infty(u_j)^{1/p}-J_\infty(u)^{1/p}|\\
&\quad+|J_{\Pi_\alpha}(u_j)^{1/p}-J_\infty(u_j)^{1/p}|.
\end{align*}
Taking first the limit in $\alpha$ and then letting $j\to\infty$ proves
\[
J_{\Pi_\alpha}(u)^{1/p}\longrightarrow J_\infty(u)^{1/p}.
\]
Thus \eqref{eq:abstract-decoupling} holds.

The implication \ref{item:allmaps-iff}$\Rightarrow$\ref{item:twovalued-iff}
is immediate because $\mu(K)<\infty$ for bounded $K$, hence
$u_K\in L^p(X;Y)$.

It remains to prove
\ref{item:twovalued-iff}$\Rightarrow$\ref{item:escape-iff}. Set
\[
c:=d_Y(y_0,y_1)>0.
\]
By the definition of $u_K$,
\begin{align*}
J_{\Pi_\alpha}(u_K)
&=c^p\bigl[
   \Pi_\alpha(K\times K^c)
   +\Pi_\alpha(K^c\times K)
  \bigr]\\
&=c^p\bigl[
   \Pi_\alpha(K\times X)
   +\Pi_\alpha(X\times K)
   -2\Pi_\alpha(K\times K)
  \bigr]\\
&=c^p\bigl[
   \int_K(a_\alpha+b_\alpha)\,\dd\mu
   -2\Pi_\alpha(K\times K)
  \bigr].
\end{align*}
On the other hand,
\[
M_p(u_K)=c^p\mu(K).
\]
Therefore, assumption~$(iii)$ and weak-* convergence, tested against
$\one_K\in L^1(X,\mu)$, give
\[
c^p\bigl[
   \int_K(a_\alpha+b_\alpha)\,\dd\mu
   -2\Pi_\alpha(K\times K)
  \bigr]
\longrightarrow
c^p\int_K(a_\infty+b_\infty)\,\dd\mu.
\]
It follows that
\[
2c^p\Pi_\alpha(K\times K)
=
c^p\int_K(a_\alpha+b_\alpha)\,\dd\mu-J_{\Pi_\alpha}(u_K)
\longrightarrow 0.
\]
Since $c>0$, we conclude that
\[
\Pi_\alpha(K\times K)\longrightarrow0,
\]
and hence $(\Pi_\alpha)_\alpha$ escapes from bounded rectangles. 
\end{proof}

\section{Target tail laws, energy profiles and localized interactions}
\label{sec:tail-laws}

The preceding theorem concerns globally $L^p$ maps. We now treat the regime
in which the map may carry a nontrivial state at infinity and
need only be $L^p$ on a bounded region of $X$.
Throughout this section $(Y,d_Y)$ is assumed to be Polish\footnote{A metric
space is called Polish if it is complete and separable.}. For
$1\le p<\infty$, let $\Pp(Y)$ denote the Borel
probability measures $\lambda$ on $Y$ with finite $p$-moment relative to
$y_0$:
\[
 \int_Y d_Y(y,y_0)^p\,\dd\lambda(y)<\infty.
\]
Let $\Wp$ be the $p$-Wasserstein distance
\begin{equation*}
 \Wp(\lambda,\eta)^p
 :=\inf_{\gamma\in\Gamma(\lambda,\eta)}
   \int_{Y\times Y}d_Y(z,w)^p\,\dd\gamma(z,w).
\end{equation*}
Here $\Gamma(\lambda,\eta)$ denotes the set of probability measures on
$Y\times Y$ whose first marginal is $\lambda$ and whose second marginal is
$\eta$.

Let $\Omega\subset X$ be a bounded Borel set with
$\mu(\Omega)>0$.  Let $(K_\varepsilon)_{\varepsilon>0}$ be a family of finite positive
\emph{Borel kernels from $\Omega$ to $X$}; that is, for every $\varepsilon>0$
and every Borel set $A\subset X$, the map
\[
x\longmapsto K_\varepsilon(x,A)
\]
is Borel measurable, while $K_\varepsilon(x,\cdot)$ is a finite positive
Borel measure on $X$ for every $x\in\Omega$.
Write
\begin{equation*}
 m_\varepsilon(x):=K_\varepsilon(x,X)<\infty.
\end{equation*}
Let $u:X\to Y$ be Borel measurable.  On the set where
$m_\varepsilon(x)>0$ define the \emph{normalized target law seen from $x$} by
\begin{equation*}
 \lambda_{\varepsilon,x}
 :=u_\#\left(\frac{K_\varepsilon(x,\cdot)}{m_\varepsilon(x)}\right),
\end{equation*}
and set $\lambda_{\varepsilon,x}:=\delta_{y_0}$ whenever
$m_\varepsilon(x)=0$.  Thus, if $m_\varepsilon(x)>0$, then for every Borel $\psi:Y\to[0,\infty]$,
\[
 \int_Y\psi(z)\,\dd\lambda_{\varepsilon,x}(z)
 =\frac1{m_\varepsilon(x)}
   \int_X\psi(u(y))\,K_\varepsilon(x,\dd y).
\]
Since $Y$ is Polish, $x\mapsto\lambda_{\varepsilon,x}$ is a Borel
probability kernel (viewed as a map into $\mathcal P(Y)$ with its weak
Borel structure); see, e.g., \cite[Proposition 7.26]{BertsekasShreve1978}. In particular, the function
$x\mapsto\Wp(\lambda_{\varepsilon,x},\lambda)^p$ is measurable for
every fixed $\lambda\in\Pp(Y)$, see e.g. \cite[Proposition 7.1.3]{AmbrosioGigliSavare2008}.

\begin{definition}
\label{def:metric-tail-law}
Let $u:X\to Y$ be Borel measurable with $u|_\Omega\in L^p(\Omega;Y)$ and
$\lambda_{\varepsilon,x}\in\Pp(Y)$ for almost every $x\in \Omega$.  A measure
$\lambda_\infty\in\Pp(Y)$ is called an \emph{$L^p$ target tail law} for the family
$K_\varepsilon$ if
\begin{equation}\label{eq:integrated-Wp-tail}
 \int_\Omega
 \Wp(\lambda_{\varepsilon,x},\lambda_\infty)^p
 \,\dd\mu(x)\longrightarrow0\quad \mbox{as}\ \varepsilon \to 0^+
\end{equation}
If, in addition, 
\begin{equation*}
 \operatorname*{ess\,sup}_{\substack{x\in\Omega\\m_\varepsilon(x)>0}}
 \frac{K_\varepsilon(x,B)}{m_\varepsilon(x)}
 \longrightarrow0 \quad \mbox{as}\ \varepsilon \to 0^+
 \qquad \forall\text{ bounded Borel }B\subset X,
\end{equation*}
we say that $\lambda_\infty$ is a \emph{target law \emph{at infinity}}.
\end{definition}

For $a\in Y$ set
\begin{equation}\label{eq:kernel-energy-profile}
 P_{\varepsilon,x}(a)
 :=\left(\int_Yd_Y(a,z)^p\,\dd\lambda_{\varepsilon,x}(z)\right)^{1/p},
 \qquad
 P_\infty(a)
 :=\left(\int_Yd_Y(a,z)^p\,\dd\lambda_\infty(z)\right)^{1/p}.
\end{equation}
We say that the target laws converge to $\lambda_\infty$ in
\emph{$p$-energy profile} on $\Omega$ if
\begin{equation}\label{eq:integrated-profile-tail}
 P_{\varepsilon,\cdot}(a)\longrightarrow P_\infty(a)
 \quad\text{in }L^p(\Omega)
 \qquad\text{for every }a\in Y.
\end{equation}
This condition records the family of costs tested by the energy and need not
determine the limiting probability measure uniquely. Since all the functions in \eqref{eq:kernel-energy-profile} are
$1$-Lipschitz in $a$, it is equivalent to require
\eqref{eq:integrated-profile-tail} only on a fixed countable dense subset of
$Y$.

\begin{lemma}
\label{lem:distance-profile-Wp}
For $a\in Y$ and $\lambda,\eta\in\Pp(Y)$  the following inequality holds
\begin{equation*}
 \left|\left(\int_Yd_Y(a,z)^p\,\dd\lambda(z)\right)^{1/p}-\left(\int_Yd_Y(a,z)^p\,\dd\eta(z)\right)^{1/p}\right|\le\Wp(\lambda,\eta).
\end{equation*}
\end{lemma}

\begin{proof}
Let $\gamma\in\Gamma(\lambda,\eta)$.  The reverse triangle inequality in
$L^p(Y\times Y,\gamma)$ and the triangle inequality give
\begin{align*}
 \left|\left(\int_Yd_Y(a,z)^p\,\dd\lambda(z)\right)^{1/p}-\left(\int_Yd_Y(a,z)^p\,\dd\eta(z)\right)^{1/p}\right|
 &\le
 \left(\int
 |d_Y(a,z)-d_Y(a,w)|^p\,\dd\gamma(z,w)\right)^{1/p}\\
 &\le
 \left(\int d_Y(z,w)^p\,\dd\gamma(z,w)\right)^{1/p}.
\end{align*}
The conclusion follows taking the infimum over $\gamma$.
\end{proof}

\begin{theorem}
\label{thm:metric-tail-law}
Let $u:X\to Y$ be Borel measurable with $u|_\Omega\in L^p(\Omega;Y)$.  Assume
that $\lambda_{\varepsilon,x}\in\Pp(Y)$ a.e. $x\in \Omega$, that
\eqref{eq:integrated-profile-tail} holds for some
$\lambda_\infty\in\Pp(Y)$, and that
\begin{equation*}
 \sup_\varepsilon\|m_\varepsilon\|_{L^\infty(\Omega)}<\infty,
 \qquad
 m_\varepsilon\longrightarrow m
 \quad\text{in }L^\infty(\Omega).
\end{equation*}
Then
\begin{equation*}
 \begin{aligned}
 &\lim_{\varepsilon\to0^+}
 \int_\Omega\int_Xd_Y(u(x),u(y))^p
 \,K_\varepsilon(x,\dd y)\,\dd\mu(x)\\
 &\qquad=
 \int_\Omega m(x)
 \int_Yd_Y(u(x),z)^p\,\dd\lambda_\infty(z)
 \,\dd\mu(x).
 \end{aligned}
\end{equation*}
\end{theorem}

\begin{proof}
For a.e. $x\in\Omega$, the maps $a\mapsto P_{\varepsilon,x}(a)$ and
$a\mapsto P_\infty(a)$ are $1$-Lipschitz.  Indeed, this follows from the
reverse triangle inequality in $L^p$ applied to the functions
$z\mapsto d_Y(a,z)$ and $z\mapsto d_Y(b,z)$.  Let $(u_j)_j$ be a sequence
of finite-valued Borel maps such that
\begin{equation}\label{eq:simple-map-approximation}
 \int_\Omega d_Y(u_j(x),u(x))^p\,\dd\mu(x)\longrightarrow0.
\end{equation}
Such a sequence exists because $Y$ is separable and
$u|_\Omega\in L^p(\Omega;Y)$.

If $u_j$ takes the values $a_{j,1},\dots,a_{j,N_j}$, then
\eqref{eq:integrated-profile-tail}, applied to this finite family, gives
\[
 P_{\varepsilon,\cdot}(u_j(\cdot))
 \longrightarrow P_\infty(u_j(\cdot))
 \quad\text{in }L^p(\Omega).
\]
The Lipschitz estimate therefore yields
\begin{align*}
 &\|P_{\varepsilon,\cdot}(u(\cdot))-P_\infty(u(\cdot))\|_{L^p(\Omega)}\\
 &\quad\le
 2\|d_Y(u,u_j)\|_{L^p(\Omega)}
 +\|P_{\varepsilon,\cdot}(u_j(\cdot))
       -P_\infty(u_j(\cdot))\|_{L^p(\Omega)}.
\end{align*}
First letting $\varepsilon\to0$ and then $j\to\infty$ proves
\begin{equation}\label{eq:Aeps-Ainf-Lp}
 A_\varepsilon:=P_{\varepsilon,\cdot}(u(\cdot))
 \longrightarrow
 A_\infty:=P_\infty(u(\cdot))
 \quad\text{in }L^p(\Omega).
\end{equation}
Moreover,
\[
 A_\infty(x)
 \le d_Y(u(x),y_0)+P_\infty(y_0),
\]
so $A_\infty\in L^p(\Omega)$ and
$\sup_\varepsilon\|A_\varepsilon\|_{L^p(\Omega)}<\infty$ for small
$\varepsilon$.

We claim that
\begin{equation*}%
 A_\varepsilon^p\longrightarrow A_\infty^p
 \quad\text{in }L^1(\Omega).
\end{equation*}
For $p=1$ this is exactly \eqref{eq:Aeps-Ainf-Lp}.  If $p>1$, then using
$|a^p-b^p|\le p|a-b|(a^{p-1}+b^{p-1})$ and the H\"older inequality we get
\begin{align*}
 \|A_\varepsilon^p-A_\infty^p\|_{L^1}
 &\le p\|A_\varepsilon-A_\infty\|_{L^p}
 \bigl(\|A_\varepsilon\|_{L^p}^{p-1}
       +\|A_\infty\|_{L^p}^{p-1}\bigr)\longrightarrow0.
\end{align*}
We claim that
\begin{align}\label{eq:kernel-profile-identity}
 \int_Xd_Y(u(x),u(y))^pK_\varepsilon(x,\dd y)
 =m_\varepsilon(x)A_\varepsilon(x)^p.
 \end{align}
Indeed, if $m_{\varepsilon}(x)=0$ then $K_\varepsilon(x,\cdot)=0$ and the equality trivially follows. If $m_{\varepsilon}(x)>0$ then
\begin{align*}
\int_Xd_Y(u(x),u(y))^pK_\varepsilon(x,\dd y)&=m_{\varepsilon}(x)\int_Xd_Y(u(x),u(y))^p\frac{K_\varepsilon(x,\dd y)}{m_{\varepsilon}(x)}\\
&=m_{\varepsilon}(x)\int_Y d_Y(u(x),z)^p\,\dd\lambda_{\varepsilon,x}(z)\\
&=m_{\varepsilon}(x)A_{\varepsilon}(x)^p.
\end{align*}
Consequently,
\begin{align*}
 &\left|\int_\Omega m_\varepsilon A_\varepsilon^p\,\dd\mu
 -\int_\Omega mA_\infty^p\,\dd\mu\right|\\
 &\qquad\le
 \|m_\varepsilon\|_\infty
 \|A_\varepsilon^p-A_\infty^p\|_{L^1}
 +\|m_\varepsilon-m\|_\infty
 \|A_\infty^p\|_{L^1}\longrightarrow0.
\end{align*}
Together with \eqref{eq:kernel-profile-identity}, this concludes the proof.
\end{proof}

\begin{corollary}
\label{cor:Wasserstein-tail-law}
Under the remaining assumptions of Theorem~\ref{thm:metric-tail-law}, the
integrated Wasserstein condition \eqref{eq:integrated-Wp-tail} implies the
conclusion of that theorem.
\end{corollary}

\begin{proof}
The condition \eqref{eq:integrated-Wp-tail} implies
\eqref{eq:integrated-profile-tail}. Indeed,
for every $a\in Y$, Lemma~\ref{lem:distance-profile-Wp} gives
\[
 |P_{\varepsilon,x}(a)-P_\infty(a)|
 \le W_p(\lambda_{\varepsilon,x},\lambda_\infty)
\]
and the conclusion follows taking the $L^p(\Omega)$ norm.
\end{proof}

\subsection{The real-valued quadratic case}

If $Y=\mathbb{R}$ and $p=2$ one can weaken the convergence of the whole target law
to convergence of only three real-valued functions.

\begin{theorem}
\label{thm:moment-only}
Let $u:X\to\mathbb R$ be measurable with $u|_\Omega\in L^2(\Omega)$.
For finite positive kernels $K_\varepsilon(x,\dd y)$ define:
\begin{align*}
 m_\varepsilon(x)&:=\int_XK_\varepsilon(x,\dd y),\\
 b_\varepsilon(x)&:=\int_Xu(y)K_\varepsilon(x,\dd y),\\
 c_\varepsilon(x)&:=\int_Xu(y)^2K_\varepsilon(x,\dd y),
\end{align*}
whenever these quantities are finite.  Assume
\[
 m_\varepsilon\to m\quad\text{in }L^\infty(\Omega),
 \qquad
 b_\varepsilon\to b\quad\text{in }L^2(\Omega),
 \qquad
 c_\varepsilon\to c\quad\text{in }L^1(\Omega).
\]
Then
\begin{equation*}
 \int_\Omega\int_X|u(x)-u(y)|^2K_\varepsilon(x,\dd y)
 \,\dd\mu(x)
 \longrightarrow
 \int_\Omega\bigl(mu^2-2bu+c\bigr)\,\dd\mu.
\end{equation*}
In particular, if $m,b,c$ are constants $M_0,M_1,M_2$, the limit can be written as
\[
 M_0\int_\Omega u^2\,\dd\mu
 -2M_1\int_\Omega u\,\dd\mu
 +M_2\mu(\Omega).
\]
\end{theorem}

\begin{proof}
For almost every $x\in X$ an explicit computation gives
\[
 \int_X|u(x)-u(y)|^2K_\varepsilon(x,\dd y)
 =m_\varepsilon(x)u(x)^2
  -2b_\varepsilon(x)u(x)+c_\varepsilon(x).
\]
The first term converges in $L^1$ by the $L^\infty$ convergence of
$m_\varepsilon$ and $u^2\in L^1(\Omega)$.  The mixed term converges in
$L^1$ by Cauchy--Schwarz, and the last term converges by assumption.
\end{proof}

\begin{remark}
\label{prop:DFP-interpretation}
Let
$\Omega\subset\mathbb R^d$ be bounded, $s>0$ and let
\[
 K_s(x,\dd y)
 :=s\,\one_{\Omega^c}(y)|x-y|^{-d-2s}\,\dd y.
\]
Following \cite{DavoliFanizzaPicerni2026}, for a measurable function $u:\mathbb{R}^d\to\mathbb{R}$ we denote by
\[
\alpha(u):=\lim_{s\to0^+}s\int_{B_1^c(0)}
\frac{u(y)}{|y|^{d+2s}}\,\dd y,
\]
whenever the limit exists and is finite. This is the mass at infinity introduced in \cite{DavoliFanizzaPicerni2026}. With this notation, we call
\[
M_0=\alpha(1)=\frac{\mathcal{H}^{d-1}(\mathbb{S}^{d-1})}{2},\qquad M_1=\alpha(u),\qquad M_2=\alpha(u^2).
\]
Their scalar hypotheses are the existence and finiteness of the relevant
masses at infinity.  In the notation of
\cite[Theorem~1.3]{DavoliFanizzaPicerni2026}, the assumptions on
$\alpha(u^+)$, $\alpha(u^-)$ and $\alpha(u^2)$ determine the three
coefficients in Theorem~\ref{thm:moment-only}.
The resulting limiting functional has the form
\begin{equation*}
 \alpha(1)\int_\Omega u^2
 -2\alpha(u)\int_\Omega u
 +\alpha(u^2)|\Omega|,
\end{equation*}
which is exactly the constant-coefficient functional appearing in
Theorem~\ref{thm:moment-only}.  When $u(r,\theta)\to u_\infty(\theta)$
at infinity, \cite[Theorem~2.11]{DavoliFanizzaPicerni2026} rewrites their
limit as
\[
 \frac12\int_\Omega\int_{\mathbb{S}^{d-1}}
 |u(x)-u_\infty(\theta)|^2
 \,\dd\mathcal H^{d-1}(\theta)\,\dd x.
\]
Writing $\sigma_{d-1}:=\mathcal H^{d-1}(\mathbb{S}^{d-1})$, this has exactly the
tail-law form with
\[
 \lambda_\infty
 =(u_\infty)_\#\left(
 \frac{\mathcal H^{d-1}\!\lfloor_{\mathbb{S}^{d-1}}}
 {\sigma_{d-1}}\right),
 \qquad
 m_\infty=\frac{\sigma_{d-1}}2.
\]
The corresponding tail law is the push-forward of normalized spherical
measure under $u_\infty$, together with total mass
$\sigma_{d-1}/2$. Its quadratic energy profile depends only on the first two
moments, while convergence in $W_2$ retains the entire probability law.

Theorem~\ref{thm:metric-tail-law} does not apply directly to the full kernel
$s\one_{\Omega^c}(y)|x-y|^{-d-2s}\dd y$, whose mass can blow
up as $x$ approaches $\partial\Omega$.  To obtain uniformly bounded kernel masses, we introduce a fixed cutoff and estimate the bounded-range contribution separately.
\end{remark}

\subsection{Angular tail laws and radial limits}

We now apply the energy-profile theorem to the Euclidean fractional kernel.
This yields a metric-valued directional Maz'ya--Shaposhnikova formula
\cite{MazyaShaposhnikova2002} for arbitrary $1\le p<\infty$ without
requiring pointwise or uniform convergence along rays.

Throughout this subsection $d\ge1$,
$\sigma_{d-1}:=\mathcal H^{d-1}(\mathbb{S}^{d-1})$, and
$\Omega\subset\mathbb R^d$ is bounded.  Let $(Y,d_Y,y_0)$ be Polish.  We write
$u\in L^p_{\mathrm{loc}}(\mathbb R^d;Y)$ if
$d_Y(u(\cdot),y_0)\in L^p_{\mathrm{loc}}(\mathbb R^d)$.  Let
$u_\infty:\mathbb{S}^{d-1}\to Y$ be Borel with
$d_Y(u_\infty(\cdot),y_0)\in L^p(\mathbb{S}^{d-1})$.  For $a\in Y$ define
the Abelian $p$-energy profile
\begin{equation}\label{eq:Abelian-p-profile}
 \mathcal A_{s,p}^u(a)
 :=s\int_{B_1^c}
 \frac{d_Y(u(y),a)^p}{|y|^{d+sp}}\,\dd y.
\end{equation}
Our tail assumption, denoted by $\mathrm{EP}_p$, is
\begin{equation}\label{eq:uniform-directional-tail}
 \mathcal A_{s,p}^u(a)
 \longrightarrow
 \frac1p\int_{\mathbb S^{d-1}}
 d_Y(a,u_\infty(\theta))^p
 \,\dd\mathcal H^{d-1}(\theta)
 \qquad\text{for every }a\in Y.
\end{equation}
Define
\begin{equation}\label{eq:directional-tail-law}
 \lambda_\infty
 :=(u_\infty)_\#\left(
 \frac{\mathcal H^{d-1}\!\lfloor_{\mathbb{S}^{d-1}}}
 {\sigma_{d-1}}\right).
\end{equation}
Then $\lambda_\infty\in \Pp(Y)$ because
\[
 \int_Y d_Y(z,y_0)^p\,\dd\lambda_\infty(z)
 =\frac1{\sigma_{d-1}}
 \int_{\mathbb S^{d-1}}d_Y(u_\infty(\theta),y_0)^p
 \,\dd\mathcal H^{d-1}(\theta)<\infty.
\]
Condition $\mathrm{EP}_p$ depends on $u_\infty$ only through the
$p$-distance profile of the push-forward law \eqref{eq:directional-tail-law}.
In particular, it does not by itself assert convergence along individual
rays, and different spherical representatives with the same profile give
the same limiting functional. Under the radial criterion below,
$u_\infty$ is also the $L^p$-limit of the radial slices.

\begin{proposition}
Let $u\in L^p_{\mathrm{loc}}(\mathbb R^d;Y)$ and suppose that
\begin{equation}\label{eq:radial-Lp-convergence}
 \int_{\mathbb S^{d-1}}
 d_Y(u(r\theta),u_\infty(\theta))^p
 \,\dd\mathcal H^{d-1}(\theta)
 \longrightarrow0
 \qquad(r\to\infty),
\end{equation}
where the convergence may be understood after changing the radial
representative on a null set of radii.  Then \eqref{eq:uniform-directional-tail}
holds.
\end{proposition}

\begin{proof}
For $a\in Y$ set
\[
 H_a(r):=\left(\int_{\mathbb S^{d-1}}
 d_Y(a,u(r\theta))^p\,\dd\mathcal H^{d-1}(\theta)\right)^{1/p}
\]
and
\[
 H_{\infty,a}:=\left(\int_{\mathbb S^{d-1}}
 d_Y(a,u_\infty(\theta))^p
 \,\dd\mathcal H^{d-1}(\theta)\right)^{1/p}.
\]
The reverse triangle inequality in $L^p(\mathbb S^{d-1})$ gives
\[
 |H_a(r)-H_{\infty,a}|
 \le
 \left(\int_{\mathbb S^{d-1}}
 d_Y(u(r\theta),u_\infty(\theta))^p
 \,\dd\mathcal H^{d-1}(\theta)\right)^{1/p},
\]
and hence $H_a(r)^p\to H_{\infty,a}^p$.  Moreover,
$H_a(\cdot)^p\in L^1_{\mathrm{loc}}([1,\infty))$ by the local
$L^p$ assumption.  The elementary fact
\[
 h(r)\longrightarrow L
 \quad\Longrightarrow\quad
 s\int_1^\infty h(r)r^{-1-sp}\,\dd r\longrightarrow\frac Lp
\]
for locally integrable, eventually bounded $h$ now gives
\[
 \mathcal A_{s,p}^u(a)
 =s\int_1^\infty H_a(r)^p r^{-1-sp}\,\dd r
 \longrightarrow\frac1pH_{\infty,a}^p,
\]
which is \eqref{eq:uniform-directional-tail}.  
\end{proof}


\begin{remark}
If $d_Y(u(\cdot),y_0)\in L^p(\mathbb R^d)$, then
$\mathrm{EP}_p$ holds with $u_\infty\equiv y_0$, or equivalently with the
tail law $\delta_{y_0}$.  Indeed, the reverse triangle inequality in the
weighted $L^p$ space gives
\begin{align*}
 &\left|
 \mathcal A_{s,p}^u(a)^{1/p}
 -\left(\frac{\sigma_{d-1}}p\right)^{1/p}d_Y(a,y_0)
 \right|\\
 &\qquad\le
 \left(s\int_{B_1^c}
 \frac{d_Y(u(y),y_0)^p}{|y|^{d+sp}}\,\dd y\right)^{1/p}
 \le s^{1/p}\|d_Y(u(\cdot),y_0)\|_{L^p(\mathbb R^d)},
\end{align*}
which tends to zero.  In particular, the condition is automatic for
globally defined fractional Sobolev maps
$u\in W^{s_0,p}(\mathbb R^d;Y)$.
\end{remark}

Put $M_\Omega:=\sup_{x\in\Omega}|x|$ and fix
$R>\max\{1,2M_\Omega\}$.  For $0<s<1$ and $x\in\Omega$ define
\begin{equation}\label{eq:far-cutoff-kernel}
 K_{s,R}(x,\dd y)
 :=s\,\one_{\{|y|>R\}}|x-y|^{-d-sp}\,\dd y,
\end{equation}
and
\[
 m_{s,R}(x):=K_{s,R}(x,\mathbb R^d)
 =s\int_{|y|>R}\frac{\dd y}{|x-y|^{d+sp}}.
\]

The following estimate is implicit in
\cite[Lemma~2.1]{DavoliFanizzaPicerni2026} and
\cite{DavoliDiFrattaGiorgioPinamonti2025}.

\begin{lemma}
\label{lem:far-angular-equidistribution}
For $\theta\in \mathbb{S}^{d-1}$ set
\[
 w_{s,R,x}(\theta)
 :=s\int_R^\infty
 \frac{r^{d-1}}{|r\theta-x|^{d+sp}}\,\dd r.
\]
Then, uniformly for $x\in\Omega$ and $\theta\in \mathbb{S}^{d-1}$,
\begin{equation}\label{eq:w-uniform}
 w_{s,R,x}(\theta)=\frac{R^{-sp}}p+O(s)
 \qquad(s\downarrow0).
\end{equation}
Consequently
\begin{equation}\label{eq:m-far-uniform}
 m_{s,R}\longrightarrow\frac{\sigma_{d-1}}p
 \quad\text{in }L^\infty(\Omega),
\end{equation}
and the normalized angular densities
\[
 g_{s,R,x}(\theta):=\frac{w_{s,R,x}(\theta)}{m_{s,R}(x)}
\]
satisfy
\begin{equation}\label{eq:g-angular-uniform}
 \sup_{x\in\Omega,\,\theta\in \mathbb{S}^{d-1}}
 \left|g_{s,R,x}(\theta)-\frac1{\sigma_{d-1}}\right|
 \longrightarrow0.
\end{equation}
\end{lemma}

\begin{proof}
Since $R>2M_\Omega$, for $r\ge R$ we have $|x|/r\le1/2$.  For
$0<s\le1$ the functions
$a\mapsto|\theta-a|^{-d-sp}$ have gradients uniformly bounded on
$|a|\le1/2$, uniformly also in $\theta\in \mathbb{S}^{d-1}$.  Hence, by the mean value theorem applied to the function $a\to |a-\theta|^{-d-sp}$ along the segment joining $0$ to $x/r$, we get
\[
 \left||\theta-x/r|^{-d-sp}-1\right|
 \le C\frac{|x|}{r}.
\]
Therefore,
\begin{align*}
 \left|w_{s,R,x}(\theta)-s\int_R^\infty r^{-1-sp}\,\dd r\right|
 &\le C M_\Omega s\int_R^\infty r^{-2-sp}\,\dd r\\
 &\le C_R s.
\end{align*}
and \eqref{eq:w-uniform} follows since
\[
 s\int_R^\infty r^{-1-sp}\,\dd r=\frac{R^{-sp}}p.
\]
Integrating in $\theta$ and using the coarea formula proves
\eqref{eq:m-far-uniform}.  Dividing \eqref{eq:w-uniform} by the mass, which
is bounded away from zero for small $s$, gives
\eqref{eq:g-angular-uniform}.
\end{proof}

\begin{lemma}
\label{lem:directional-Wp-tail}
Let $u\in L^p_{\mathrm{loc}}(\mathbb R^d;Y)$ satisfy
\eqref{eq:uniform-directional-tail}.  Let $\lambda_{s,R,x}$ be the
normalized target law associated with \eqref{eq:far-cutoff-kernel}.  Then
$\lambda_{s,R,x}\in\Pp(Y)$ and, for every $a\in Y$,
\begin{equation}\label{eq:directional-Wp-uniform}
 \sup_{x\in\Omega}
 \left|
 \left(\int_Yd_Y(a,z)^p\,\dd\lambda_{s,R,x}(z)\right)^{1/p}
 -\left(\int_Yd_Y(a,z)^p\,\dd\lambda_\infty(z)\right)^{1/p}
 \right|
 \longrightarrow0.
\end{equation}
\end{lemma}

\begin{proof}
For $a\in Y$ set
\[
 N_{s,R,x}(a)
 :=s\int_{|y|>R}
 \frac{d_Y(a,u(y))^p}{|x-y|^{d+sp}}\,\dd y.
\]
Since $R>1$, local $L^p$ integrability gives
\[
 s\int_{1<|y|\le R}
 \frac{d_Y(a,u(y))^p}{|y|^{d+sp}}\,\dd y\longrightarrow0.
\]
Consequently, \eqref{eq:uniform-directional-tail} implies
\begin{equation}\label{eq:origin-profile-cutoff}
 N_{s,R,0}(a)
 \longrightarrow
 L(a):=\frac1p\int_{\mathbb S^{d-1}}
 d_Y(a,u_\infty(\theta))^p
 \,\dd\mathcal H^{d-1}(\theta).
\end{equation}

We next show that translating the centre by a bounded vector does not alter
this limit.  Fix $A>R$.  On $R<|y|\le A$, both $|y|$ and $|x-y|$ are
bounded away from zero, uniformly for $x\in\Omega$.  Hence local
$L^p$ integrability gives
\begin{equation}\label{eq:A-normalized-vanish}
 \sup_{x\in\Omega}s\int_{R<|y|\le A}d_Y(a,u(y))^p
 \left||x-y|^{-d-sp}-|y|^{-d-sp}\right|\,\dd y
 \longrightarrow0.
\end{equation}
For $|y|>A>2M_\Omega$, the mean value theorem yields, uniformly for
$0<s\le1$,
\[
 \left||x-y|^{-d-sp}-|y|^{-d-sp}\right|
 \le C\frac{M_\Omega}{|y|}|y|^{-d-sp}.
\]
By \eqref{eq:origin-profile-cutoff}, there exists $s_a\in(0,1)$ such that
$\sup_{0<s<s_a}N_{s,R,0}(a)<\infty$.  It follows that
\begin{equation}\label{eq:B-normalized-estimate}
 \limsup_{s\downarrow0}\sup_{x\in\Omega}
 |N_{s,R,x}(a)-N_{s,R,0}(a)|
 \le \frac{C M_\Omega}{A}
 \sup_{0<s<s_a}N_{s,R,0}(a).
\end{equation}
Letting $A\to\infty$ and using \eqref{eq:origin-profile-cutoff}, we obtain
\begin{equation}\label{eq:shifted-profile-uniform}
 \sup_{x\in\Omega}|N_{s,R,x}(a)-L(a)|\longrightarrow0.
\end{equation}
In particular, the case $a=y_0$ shows that
$\lambda_{s,R,x}\in\Pp(Y)$ for small $s$.  Finally,
\[
 \int_Yd_Y(a,z)^p\,\dd\lambda_{s,R,x}(z)
 =\frac{N_{s,R,x}(a)}{m_{s,R}(x)}.
\]
Combining \eqref{eq:shifted-profile-uniform} with
\eqref{eq:m-far-uniform}, and observing that
\[
 \int_Yd_Y(a,z)^p\,\dd\lambda_\infty(z)
 =\frac1{\sigma_{d-1}}
 \int_{\mathbb S^{d-1}}d_Y(a,u_\infty(\theta))^p
 \,\dd\mathcal H^{d-1}(\theta),
\]
proves \eqref{eq:directional-Wp-uniform}.
\end{proof}

For $0<s<1$ define the \emph{localized fractional interaction}
\begin{equation*}
 \mathcal E_{s,p}^\Omega(u)
 :=\iint_{Q_\Omega}
 \frac{d_Y(u(x),u(y))^p}{|x-y|^{d+sp}}\,\dd y\,\dd x,
 \qquad
 Q_\Omega=(\mathbb R^d\times\mathbb R^d)
 \setminus(\Omega^c\times\Omega^c).
\end{equation*}

\begin{theorem}
\label{thm:directional-metric-MS}
Let $1\le p<\infty$, let $u\in L^p_{\mathrm{loc}}(\mathbb R^d;Y)$, and let
$u_\infty:\mathbb S^{d-1}\to Y$ be Borel with
$d_Y(u_\infty(\cdot),y_0)\in L^p(\mathbb S^{d-1})$.  Assume
$\mathrm{EP}_p$, namely \eqref{eq:uniform-directional-tail}.  Suppose that
for some $s_0\in(0,1)$,
\begin{equation*}
 \mathcal E_{s_0,p}^\Omega(u)<\infty.
\end{equation*}
Then
\begin{equation}\label{eq:directional-metric-MS}
 \lim_{s\downarrow0}\frac{s}{2}\mathcal E_{s,p}^\Omega(u)
 =\frac1p\int_\Omega\int_{\mathbb{S}^{d-1}}
 d_Y(u(x),u_\infty(\theta))^p
 \,\dd\mathcal H^{d-1}(\theta)\,\dd x.
\end{equation}
\end{theorem}

\begin{proof}
Fix $R>\max\{1,2M_\Omega\}$.  By symmetry,
\begin{align*}
 \frac{s}{2}\mathcal E_{s,p}^\Omega(u)
 &=\frac{s}{2}\iint_{\Omega\times\Omega}
 \frac{d_Y(u(x),u(y))^p}{|x-y|^{d+sp}}\,\dd y\,\dd x\quad (=:\mathcal{A}_s)\\
 &\quad+s\int_\Omega\int_{\Omega^c\cap B_R}
 \frac{d_Y(u(x),u(y))^p}{|x-y|^{d+sp}}\,\dd y\,\dd x\quad (=:\mathcal{B}_s)\\
 &\quad+s\int_\Omega\int_{\Omega^c\cap \{|y|>R\}}
 \frac{d_Y(u(x),u(y))^p}{|x-y|^{d+sp}}\,\dd y\,\dd x.\quad (=:\mathcal{C}_s)
\end{align*}
Now, splitting at $|x-y|=1$ and using $s<s_0$ on the short-range part,
\begin{align*}
\mathcal A_s
&\le \frac{s}{2}\mathcal E_{s_0,p}^\Omega(u)
 +2^{p-1}s\,|\Omega|
 \int_\Omega d_Y(u(x),y_0)^p\,\dd x
 \longrightarrow0.
\end{align*}
Similarly, with $D_R:=\Omega^c\cap B_R$,
\begin{align*}
\mathcal B_s
&\le s\mathcal E_{s_0,p}^\Omega(u)
 +2^{p-1}s\,|\Omega|
 \int_{D_R}d_Y(u(y),y_0)^p\,\dd y\\
&\quad
 +2^{p-1}s\,|D_R|
 \int_\Omega d_Y(u(x),y_0)^p\,\dd x
 \longrightarrow0.
\end{align*}
For the last term use the kernel \eqref{eq:far-cutoff-kernel}.  By
Lemma~\ref{lem:far-angular-equidistribution},
$m_{s,R}\to\sigma_{d-1}/p$ in $L^\infty(\Omega)$, and, in particular,
\[
\sup_{0<s<s_*}\|m_{s,R}\|_{L^\infty(\Omega)}<\infty
\]
for some $s_*>0$.  Lemma~\ref{lem:directional-Wp-tail} gives the
$p$-energy-profile convergence \eqref{eq:integrated-profile-tail}, in fact
uniformly in $x$.  Therefore Theorem~\ref{thm:metric-tail-law}
gives\footnote{Since $R>2M_\Omega$, we have
$\{|y|>R\}\subset\Omega^c$ and, up to the null sphere $\partial B_R$,
$\Omega^c=(\Omega^c\cap B_R)\cup\{|y|>R\}$.}
\begin{align*}
 \lim_{s\downarrow 0} \mathcal C_s&=\lim_{s\downarrow0}s\int_\Omega\int_{\{|y|>R\}}
 \frac{d_Y(u(x),u(y))^p}{|x-y|^{d+sp}}\,\dd y\,\dd x\\
 &=\frac{\sigma_{d-1}}p
 \int_\Omega\int_Yd_Y(u(x),z)^p\,\dd\lambda_\infty(z)\,\dd x\\
 &=\frac1p\int_\Omega\int_{\mathbb{S}^{d-1}}
 d_Y(u(x),u_\infty(\theta))^p
 \,\dd\mathcal H^{d-1}(\theta)\,\dd x.
\end{align*}
Combining the three contributions proves \eqref{eq:directional-metric-MS}.
\end{proof}
\begin{remark}[Comparison with \cite{DavoliFanizzaPicerni2026}]
For $Y=\mathbb R$ and $p=2$, Theorem~\ref{thm:directional-metric-MS} gives
\[
 \lim_{s\downarrow0}\frac{s}{2}\mathcal E_{s,2}^\Omega(u)
 =\frac12\int_\Omega\int_{\mathbb{S}^{d-1}}
 |u(x)-u_\infty(\theta)|^2
 \,\dd\mathcal H^{d-1}(\theta)\,\dd x,
\]
which is the directional formula of
\cite[Theorem~2.11]{DavoliFanizzaPicerni2026}.  More precisely, writing
\[
 \alpha_s(f):=s\int_{B_1^c}\frac{f(y)}{|y|^{d+2s}}\,\dd y,
\]
we have
\[
 \mathcal A_{s,2}^u(a)
 =\frac{\sigma_{d-1}}2a^2-2a\alpha_s(u)+\alpha_s(u^2).
\]
Consequently $\mathrm{EP}_2$ is equivalent to
\[
 \alpha_s(u)\longrightarrow
 \frac12\int_{\mathbb S^{d-1}}u_\infty\,\dd\mathcal H^{d-1},
 \qquad
 \alpha_s(u^2)\longrightarrow
 \frac12\int_{\mathbb S^{d-1}}u_\infty^2\,\dd\mathcal H^{d-1}.
\]
These are precisely the first two moments entering the quadratic
mass-at-infinity formula.  The hypotheses of
\cite[Theorem~2.11]{DavoliFanizzaPicerni2026}, through
\cite[Lemma~2.9]{DavoliFanizzaPicerni2026}, imply these identities, but also
require bounded exterior data and uniform radial convergence. Thus the
present hypotheses give their pointwise directional formula and also cover
tails specified only through energy profiles. For arbitrary $p$ and Polish targets,
the energy-profile condition replaces the quadratic moment expansion by the
family of costs $a\mapsto d_Y(a,\cdot)^p$.
\end{remark}

\subsubsection{Anisotropic consequences}

The preceding directional theorem is stable under replacing the Euclidean norm by a norm induced by a symmetric convex body. This yields the following anisotropic form and, as a scalar specialization, Ludwig's limit \cite[Theorem~2]{Ludwig2014} when the state at infinity vanishes.

\begin{proposition}
\label{prop:anisotropic-directional-MS}
Let $K\subset\mathbb R^n$ be an origin-symmetric convex body with nonempty
interior, and let $\|\cdot\|_K$ be its Minkowski norm. Set
\[
 S_K:=\{\omega\in\mathbb R^n:\|\omega\|_K=1\},
\]
and let $\sigma_K$ be the polar measure on $S_K$, characterized by
\begin{equation}\label{eq:K-polar-formula}
 \int_{\mathbb R^n} f(z)\,\dd z
 =
 \int_0^\infty\int_{S_K}
 f(r\omega)\,\dd\sigma_K(\omega)\,r^{n-1}\,\dd r
\end{equation}
for every nonnegative Borel function $f$. In particular,
\begin{equation*}
 \sigma_K(S_K)=n|K|.
\end{equation*}

Let $(Y,d_Y,y_0)$ be a Polish pointed metric space,
let $\Omega\subset\mathbb R^n$ be bounded and Borel, and let
$u\in L^p_{\rm loc}(\mathbb R^n;Y)$, where $1\le p<\infty$.
Assume that there exists a Borel map
\[
 u_\infty:S_K\to Y,
 \qquad
 \int_{S_K}d_Y(u_\infty(\omega),y_0)^p
 \,\dd\sigma_K(\omega)<\infty,
\]
such that the anisotropic energy-profile condition
\begin{equation}\label{eq:anisotropic-energy-profile}
 s\int_{\{\|y\|_K>1\}}
 \frac{d_Y(u(y),a)^p}{\|y\|_K^{n+sp}}\,\dd y
 \longrightarrow
 \frac1p\int_{S_K}d_Y(a,u_\infty(\omega))^p\,\dd\sigma_K(\omega)
 \qquad\text{for every }a\in Y
\end{equation}
holds.

Define
\begin{equation*}
 \mathcal E_{s,p}^{\Omega,K}(u)
 :=
 \iint_{Q_\Omega}
 \frac{d_Y(u(x),u(y))^p}
      {\|x-y\|_K^{n+sp}}
 \,\dd y\,\dd x,
 \qquad
 Q_\Omega
 :=
 (\mathbb R^n\times\mathbb R^n)
 \setminus(\Omega^c\times\Omega^c).
\end{equation*}
If
\begin{equation}\label{eq:anisotropic-positive-order}
 \mathcal E_{s_0,p}^{\Omega,K}(u)<\infty
\end{equation}
for some $s_0\in(0,1)$, then
\begin{equation}\label{eq:anisotropic-directional-MS}
 \lim_{s\downarrow0}
 \frac{s}{2}\mathcal E_{s,p}^{\Omega,K}(u)
 =
 \frac1p
 \int_\Omega\int_{S_K}
 d_Y\bigl(u(x),u_\infty(\omega)\bigr)^p
 \,\dd\sigma_K(\omega)\,\dd x.
\end{equation}
\end{proposition}

\begin{proof}
Set
\[
 M_{\Omega,K}:=\sup_{x\in\Omega}\|x\|_K
\]
and fix $R>\max\{1,2M_{\Omega,K}\}$. For $x\in\Omega$ consider the far kernel
\[
 K^K_{s,R}(x,\dd y)
 :=
 s\,\one_{\{\|y\|_K>R\}}
 \|x-y\|_K^{-n-sp}\,\dd y.
\]
Using the polar formula \eqref{eq:K-polar-formula}, its angular density is
\[
 w^K_{s,R,x}(\omega)
 :=
 s\int_R^\infty
 \frac{r^{n-1}}
      {\|r\omega-x\|_K^{n+sp}}
 \,\dd r.
\]
Since
\[
 \|r\omega-x\|_K
 =
 r\|\omega-x/r\|_K
\]
and
\[
 \left|
 \|\omega-x/r\|_K-1
 \right|
 \le \frac{\|x\|_K}{r}
 \le\frac12,
\]
the function $t\mapsto t^{-n-sp}$ is uniformly Lipschitz on the relevant
compact interval, uniformly for $0<s\le1$. Hence
\[
 \left|
 \|\omega-x/r\|_K^{-n-sp}-1
 \right|
 \le C\frac{\|x\|_K}{r},
\]
and therefore, uniformly for $x\in\Omega$ and $\omega\in S_K$,
\begin{equation}\label{eq:anisotropic-w-asymptotic}
 w^K_{s,R,x}(\omega)
 =
 s\int_R^\infty r^{-1-sp}\,\dd r+O(s)
 =
 \frac{R^{-sp}}p+O(s).
\end{equation}
Consequently, if
\[
 m^K_{s,R}(x):=K^K_{s,R}(x,\mathbb R^n),
\]
then
\begin{equation}\label{eq:anisotropic-far-mass}
 m^K_{s,R}
 \longrightarrow
 \frac{\sigma_K(S_K)}p
 =
 \frac{n|K|}{p}
 \qquad\text{in }L^\infty(\Omega),
\end{equation}
and the corresponding normalized angular densities converge uniformly to
\[
 \frac1{\sigma_K(S_K)}=\frac1{n|K|}.
\]
Define
\begin{equation*}
 \lambda_\infty^K
 :=
 (u_\infty)_\#
 \left(
 \frac{\sigma_K}{n|K|}
 \right).
\end{equation*}
The assumed $p$-integrability of $u_\infty$ gives
$\lambda_\infty^K\in\Pp(Y)$.
Let $\lambda^K_{s,R,x}$ denote the normalized target law of
$K^K_{s,R}(x,\cdot)$.  We claim that, for every $a\in Y$,
\begin{equation}\label{eq:anisotropic-profile-uniform}
 \sup_{x\in\Omega}
 \left|
 \left(\int_Yd_Y(a,z)^p\,\dd\lambda^K_{s,R,x}(z)\right)^{1/p}
 -\left(\int_Yd_Y(a,z)^p\,\dd\lambda_\infty^K(z)\right)^{1/p}
 \right|
 \longrightarrow0.
\end{equation}
Indeed, set
\[
 N^K_{s,R,x}(a)
 :=s\int_{\{\|y\|_K>R\}}
 \frac{d_Y(a,u(y))^p}{\|x-y\|_K^{n+sp}}\,\dd y.
\]
Condition \eqref{eq:anisotropic-energy-profile} and local $L^p$
integrability on the fixed annulus $1<\|y\|_K\le R$ give
\[
 N^K_{s,R,0}(a)
 \longrightarrow
 \frac1p\int_{S_K}d_Y(a,u_\infty(\omega))^p\,\dd\sigma_K(\omega).
\]
On every fixed annulus $R<\|y\|_K\le A$, the difference between the kernels
centred at $x$ and at the origin contributes $o(1)$, uniformly in
$x\in\Omega$.  For $\|y\|_K>A>2M_{\Omega,K}$, the triangle inequality for
$\|\cdot\|_K$ and the mean value theorem give
\[
 \left|
 \|x-y\|_K^{-n-sp}-\|y\|_K^{-n-sp}
 \right|
 \le C\frac{M_{\Omega,K}}{\|y\|_K}
 \|y\|_K^{-n-sp}.
\]
The same cutoff argument as in the proof of
Lemma~\ref{lem:directional-Wp-tail} therefore yields
\[
 \sup_{x\in\Omega}
 \left|N^K_{s,R,x}(a)-N^K_{s,R,0}(a)\right|\longrightarrow0.
\]
After division by $m^K_{s,R}(x)$ and use of
\eqref{eq:anisotropic-far-mass}, this proves
\eqref{eq:anisotropic-profile-uniform}; the case $a=y_0$ also shows that
$\lambda^K_{s,R,x}\in\Pp(Y)$ for small $s$.
By symmetry of $\|\cdot\|_K$,
\[
 \frac{s}{2}\mathcal E_{s,p}^{\Omega,K}(u)
 =
 \frac{s}{2}
 \iint_{\Omega\times\Omega}
 \frac{d_Y(u(x),u(y))^p}{\|x-y\|_K^{n+sp}}
 \,\dd y\,\dd x
\]
\[
 \qquad
 +s\int_\Omega\int_{\Omega^c\cap\{\|y\|_K<R\}}
 \frac{d_Y(u(x),u(y))^p}{\|x-y\|_K^{n+sp}}
 \,\dd y\,\dd x
 +s\int_\Omega\int_{\{\|y\|_K>R\}}
 \frac{d_Y(u(x),u(y))^p}{\|x-y\|_K^{n+sp}}
 \,\dd y\,\dd x .
\]
The first two terms tend to zero. Indeed, on $\|x-y\|_K<1$ this follows
from \eqref{eq:anisotropic-positive-order}, while on
$\|x-y\|_K\ge1$ both variables range in a bounded set and local
$L^p$ integrability provides an $s$-independent integrable bound.
For the far term, Theorem~\ref{thm:metric-tail-law},
\eqref{eq:anisotropic-far-mass}, and
\eqref{eq:anisotropic-profile-uniform} give
\begin{align*}
 &\lim_{s\downarrow0}
 s\int_\Omega\int_{\{\|y\|_K>R\}}
 \frac{d_Y(u(x),u(y))^p}{\|x-y\|_K^{n+sp}}
 \,\dd y\,\dd x
 \\
 &\qquad=
 \frac{n|K|}{p}
 \int_\Omega\int_Y
 d_Y(u(x),z)^p
 \,\dd\lambda_\infty^K(z)\,\dd x
 \\
 &\qquad=
 \frac1p
 \int_\Omega\int_{S_K}
 d_Y\bigl(u(x),u_\infty(\omega)\bigr)^p
 \,\dd\sigma_K(\omega)\,\dd x.
\end{align*}
This proves \eqref{eq:anisotropic-directional-MS}.
\end{proof}

\begin{corollary}
\label{cor:scalar-anisotropic-directional-MS}
Let $(Y,d_Y,y_0)=(\mathbb R,|\cdot|,0)$. Under the assumptions of Proposition \ref{prop:anisotropic-directional-MS} the following formula holds:
\begin{equation}\label{eq:scalar-anisotropic-directional-MS}
\lim_{s\downarrow0}
\frac{s}{2}
\iint_{Q_\Omega}
\frac{|u(x)-u(y)|^p}
     {\|x-y\|_K^{n+sp}}
\,\dd y\,\dd x
=
\frac1p
\int_\Omega\int_{S_K}
|u(x)-u_\infty(\omega)|^p
\,\dd\sigma_K(\omega)\,\dd x.
\end{equation}
\end{corollary}

\begin{remark}
If $u_\infty\equiv0$, then \eqref{eq:scalar-anisotropic-directional-MS}
reduces to
\[
\lim_{s\downarrow0}
\frac{s}{2}
\iint_{Q_\Omega}
\frac{|u(x)-u(y)|^p}
     {\|x-y\|_K^{n+sp}}
\,\dd y\,\dd x
=
\frac{n|K|}{p}
\int_\Omega |u(x)|^p\,\dd x,
\]
since $\sigma_K(S_K)=n|K|$. When $u$ is compactly supported in $\Omega$,
the formula is Ludwig's anisotropic Maz'ya--Shaposhnikova theorem
\cite[Theorem~2]{Ludwig2014}. The localized statement also allows
nontrivial exterior data and retains the $p$-distance profile of the
anisotropic distribution at infinity.
\end{remark}
\section{Abstract small-order nonlocal energies}\label{abstract}
We formulate the short-range/far-range decomposition for singular nonlocal
energies. Let $\varepsilon\downarrow0$ be a parameter and let
$\Lambda_\varepsilon$ be a positive Borel measure on $X\times X$.  Define
\begin{equation*}
 \mathcal E_\varepsilon(u)
 :=\int_{X\times X}d_Y(u(x),u(y))^p\,\dd\Lambda_\varepsilon(x,y).
\end{equation*}
We assume that
\begin{equation}\label{eq:near-far-decomposition}
 \Lambda_\varepsilon
 =\Lambda_\varepsilon^{\mathrm{near}}
 +\Lambda_\varepsilon^{\mathrm{far}}
\end{equation}
with both terms positive, and let $c_\varepsilon>0$ be a normalization factor chosen so that the rescaled energies
$c_\varepsilon \mathcal{E}_\varepsilon$ admit a nontrivial limit as $\varepsilon \downarrow0$.

\begin{proposition}
\label{thm:near-far}
Let $u\in L^p(X;Y)$.  Assume:
\begin{enumerate}[label=\textup{(A\arabic*)}]
\item \label{ass:near}
\begin{equation}\label{eq:minimal-near}
 c_\varepsilon
 \int_{X\times X}d_Y(u(x),u(y))^p
 \,\dd\Lambda_\varepsilon^{\mathrm{near}}(x,y)
 \longrightarrow0;
\end{equation}
\item \label{ass:far-marginals}
$c_\varepsilon\Lambda_\varepsilon^{\mathrm{far}}$ has marginals
$a_\varepsilon\mu$ and $b_\varepsilon\mu$, where
$a_\varepsilon,b_\varepsilon,a_\infty,b_\infty\in L^\infty_+(X,\mu)$ and
\[
 \sup_\varepsilon
 \bigl(\|a_\varepsilon\|_\infty+\|b_\varepsilon\|_\infty\bigr)<\infty,
 \qquad
 a_\varepsilon\stackrel{*}{\rightharpoonup}a_\infty,
 \qquad
 b_\varepsilon\stackrel{*}{\rightharpoonup}b_\infty
 \quad\text{in }L^\infty(X,\mu);
\]
\item \label{ass:far-escape}
$c_\varepsilon\Lambda_\varepsilon^{\mathrm{far}}$ escapes from bounded
rectangles.
\end{enumerate}
Then
\begin{equation}\label{eq:near-far-limit}
 c_\varepsilon\mathcal E_\varepsilon(u)
 \longrightarrow
 \int_X(a_\infty+b_\infty)d_Y(u,y_0)^p\,\dd\mu.
\end{equation}
\end{proposition}

\begin{proof}
By \eqref{eq:near-far-decomposition},
\begin{align*}
 c_\varepsilon\mathcal E_\varepsilon(u)
 =c_\varepsilon
 \int d_Y(u(x),u(y))^p\,\dd\Lambda_\varepsilon^{\mathrm{near}} +
 \int d_Y(u(x),u(y))^p
 \,\dd(c_\varepsilon\Lambda_\varepsilon^{\mathrm{far}}).
\end{align*}
The first term tends to zero by \eqref{eq:minimal-near}. The second tends
to the right-hand side of \eqref{eq:near-far-limit} by
Theorem~\ref{thm:abstract-decoupling}.  This proves \eqref{eq:near-far-limit}.
\end{proof}

Besides $u\in L^p(X;Y)$, the only regularity assumption is
\eqref{eq:minimal-near}; belonging to a specified Sobolev space is not
required. Fractional Sobolev regularity at one positive order is sufficient
for a large class of kernels.

\subsection{A general fractional sufficient condition}

Let $k:X\times X\to[0,\infty]$ be a Borel function independent of
$s$; no symmetry is required.  For $0<s<1$ define the functional
\begin{equation}\label{eq:k-fractional}
 [u]_{s,p;k,<1}^p
 :=\int_{\{0<d_X(x,y)<1\}}
 \frac{d_Y(u(x),u(y))^p}{d_X(x,y)^{sp}}
 k(x,y)\,\dd\mu(y)\,\dd\mu(x).
\end{equation}

\begin{proposition}
\label{prop:fractional-general}
Assume that for some $0<\theta<1$,
\begin{equation}\label{eq:theta-finite-general}
 [u]_{\theta,p;k,<1}<\infty.
\end{equation}
Then
\begin{equation*}
 \lim_{s\downarrow0}s[u]_{s,p;k,<1}^p=0.
\end{equation*}
\end{proposition}

\begin{proof}
For $0<s<\theta$ and $0<d_X(x,y)<1$,
\[
 d_X(x,y)^{-sp}\le d_X(x,y)^{-\theta p}.
\]
Therefore
\[
 0\le s[u]_{s,p;k,<1}^p
 \le s[u]_{\theta,p;k,<1}^p\longrightarrow0.
\]
\end{proof}
\begin{remark}
Proposition~\ref{prop:fractional-general} is stated under the local
assumption \eqref{eq:theta-finite-general}, since only the interaction near
the diagonal is relevant for \eqref{eq:minimal-near}. In many standard
situations one may impose the stronger global assumption

$$
u\in L^p(X;Y), \qquad
\int_{X\times X}
\frac{d_Y(u(x),u(y))^p}{d_X(x,y)^{\theta p}}
\,k(x,y)\,d\mu(y)\,d\mu(x)<\infty.
$$

This hypothesis implies \eqref{eq:theta-finite-general}, and hence the
conclusion of Proposition~\ref{prop:fractional-general}. For the usual
Gagliardo-type kernels on spaces with suitable volume growth, the
contribution of \(d_X(x,y)\ge1\) is controlled by the \(L^p\)-norm of
\(u\). In that setting the global formulation is equivalent, up to the
\(L^p\) requirement, to the local assumption above.
\end{remark}
\begin{remark}
A common metric-measure choice, for $x\ne y$, is
\begin{equation*}
 k_X(x,y)
 :=\frac12\left(
 \frac1{\mu(B(x,d_X(x,y)))}
 +\frac1{\mu(B(y,d_X(x,y)))}
 \right),
\end{equation*}
and one may set $k_X(x,x)=0$.  This formula is used whenever the ball
measures are positive and finite.  On a $Q$-Ahlfors
regular space this kernel is comparable to $d_X(x,y)^{-Q}$, so
\eqref{eq:k-fractional} is comparable with the usual Gagliardo kernel
$d_X(x,y)^{-Q-sp}$. 
\end{remark}
\subsection{Application: fractional perimeters as Bernoulli tail laws}

For characteristic functions, the target law is a Bernoulli law. The
following proposition identifies its parameter with the mass at infinity.

\begin{proposition}
\label{prop:Bernoulli-tail-perimeter}
Let $\Omega\subset\mathbb R^n$ be bounded and open and let
$E\subset\mathbb R^n$ be measurable.  For $0<s<1$ write
\[
 L_s(A,B)
 :=
 \int_A\int_B\frac{1}{|x-y|^{n+s}}\,\dd y\,\dd x
\]
and
\[
 \operatorname{Per}_s(E;\Omega)
 :=
 L_s(E\cap\Omega,E^c\cap\Omega)
 +L_s(E\cap\Omega,E^c\cap\Omega^c)
 +L_s(E\cap\Omega^c,E^c\cap\Omega).
\]
Assume that
\[
 \operatorname{Per}_{s_0}(E;\Omega)<\infty
\]
for some $s_0\in(0,1)$ and that the limit
\begin{equation}\label{eq:alpha-E-tail}
 \alpha(E)
 :=
 \lim_{s\downarrow0}
 s\int_{E\cap B_1^c}\frac{\dd y}{|y|^{n+s}}
\end{equation}
exists.  Set
\[
 \sigma_{n-1}:=\mathcal H^{n-1}(S^{n-1}),
 \qquad
 \vartheta_E:=\frac{\alpha(E)}{\sigma_{n-1}}\in[0,1].
\]
Then
\begin{equation}\label{eq:DFPV-Bernoulli-limit}
 \lim_{s\downarrow0}
 s\,\operatorname{Per}_s(E;\Omega)
 =
 \sigma_{n-1}
 \left[
 (1-\vartheta_E)|E\cap\Omega|
 +
 \vartheta_E|E^c\cap\Omega|
 \right].
\end{equation}
Equivalently, if
\begin{equation}\label{eq:Bernoulli-tail-law}
 \lambda_\infty^E
 :=
 (1-\vartheta_E)\delta_0+\vartheta_E\delta_1,
\end{equation}
then
\begin{equation}\label{eq:DFPV-tail-law-form}
 \lim_{s\downarrow0}
 s\,\operatorname{Per}_s(E;\Omega)
 =
 \sigma_{n-1}
 \int_\Omega\int_{\{0,1\}}
 |\one_E(x)-z|
 \,\dd\lambda_\infty^E(z)\,\dd x.
\end{equation}
In particular, the mass at infinity appearing in the fractional-perimeter
formula of Dipierro--Figalli--Palatucci--Valdinoci
\cite[Theorem~2.5]{DipierroFigalliPalatucciValdinoci2013}
is precisely the parameter of a Bernoulli target law at infinity.
\end{proposition}

\begin{proof}
Observe first that
\begin{equation}\label{eq:perimeter-characteristic-energy}
 \iint_{Q_\Omega}
 \frac{|\one_E(x)-\one_E(y)|}{|x-y|^{n+s}}
 \,\dd y\,\dd x
 =
 2\,\operatorname{Per}_s(E;\Omega).
\end{equation}
Fix $R>1$ so large that
$R>2\sup_{x\in\Omega}|x|$.
For $x\in\Omega$ consider the kernel
\[
 K_{s,R}(x,\dd y)
 :=
 s\,\one_{\{|y|>R\}}|x-y|^{-n-s}\,\dd y.
\]
As in Lemma~\ref{lem:far-angular-equidistribution},
\begin{equation}\label{eq:perimeter-far-mass}
 K_{s,R}(x,\mathbb R^n)
 \longrightarrow\sigma_{n-1}
\end{equation}
uniformly in $x\in\Omega$.
Since $R$ is fixed,
\[
 s\int_{E\cap\{1<|y|\le R\}}
 |y|^{-n-s}\,\dd y\longrightarrow0,
\]
and hence \eqref{eq:alpha-E-tail} implies
\begin{equation*}
 s\int_{E\cap\{|y|>R\}}
 |y|^{-n-s}\,\dd y
 \longrightarrow\alpha(E).
\end{equation*}
Moreover, uniformly for $x\in\Omega$,
\[
 s\int_{\{|y|>R\}}
 \left|
 |x-y|^{-n-s}-|y|^{-n-s}
 \right|
 \,\dd y
 \longrightarrow0.
\]
Indeed, for $|y|>R$ the mean value theorem and
$|x|\le R/2$ give
\[
 \left|
 |x-y|^{-n-s}-|y|^{-n-s}
 \right|
 \le C|x|\,|y|^{-n-s-1},
\]
and the last expression is integrable with an additional factor $s$.
Consequently,
\begin{equation}\label{eq:E-far-mass-alpha}
 \sup_{x\in\Omega}
 \left|
 s\int_{E\cap\{|y|>R\}}
 |x-y|^{-n-s}\,\dd y-\alpha(E)
 \right|
 \longrightarrow0.
\end{equation}

The normalized push-forward of $K_{s,R}(x,\cdot)$ under $\one_E$ is a Bernoulli probability measure. By
\eqref{eq:perimeter-far-mass} and \eqref{eq:E-far-mass-alpha},
its parameter converges uniformly to
$\vartheta_E=\alpha(E)/\sigma_{n-1}$; hence the corresponding target
laws converge to \eqref{eq:Bernoulli-tail-law}.

It remains only to remove the far cutoff.  On
$|x-y|<1$, for $0<s<s_0$,
\[
 |x-y|^{-n-s}\le |x-y|^{-n-s_0},
\]
and therefore the bounded-range contribution, multiplied by $s$, tends
to zero by the assumption
$\operatorname{Per}_{s_0}(E;\Omega)<\infty$.  On the remaining bounded
region $|x-y|\ge1$, the same conclusion follows directly from boundedness
of the integration region.

Applying Theorem~\ref{thm:metric-tail-law} with the target
$Y=\{0,1\}$, and using \eqref{eq:perimeter-characteristic-energy}, gives
\eqref{eq:DFPV-tail-law-form}.  Finally,
\[
 \int_{\{0,1\}}|\one_E(x)-z|\,\dd\lambda_\infty^E(z)
 =
 (1-\vartheta_E)\one_E(x)
 +\vartheta_E\one_{E^c}(x),
\]
which yields \eqref{eq:DFPV-Bernoulli-limit}.
\end{proof}

\section{Escaping transition kernels and semigroup Besov energies}
\label{sec:semigroups}
\noindent We now specialize the abstract coupling theorem to
transition kernels. Let
\[
q:(0,\infty)\times X\times X\to[0,\infty)
\]
be jointly Borel measurable, and write
$q_t(x,y):=q(t,x,y)$. For every $t>0$, define
the measure
\begin{equation}\label{eq:Pi-qt}
 \dd\Pi_t(x,y):=q_t(x,y)\,\dd\mu(x)\,\dd\mu(y).
\end{equation}
For $t\ge1$, define the marginal densities
\begin{equation*}
 a_t(x):=\int_Xq_t(x,y)\,\dd\mu(y),
 \qquad
 b_t(y):=\int_Xq_t(x,y)\,\dd\mu(x),
\end{equation*}
and assume that $a_t,b_t\in L^\infty_+(X,\mu)$. Then $\Pi_t$ is an
$(a_t,b_t)$-marginal coupling.
For $u\in L^p(X;Y)$ set
\begin{equation*}
 J_p(t,u):=
 \int_X\int_Xd_Y(u(x),u(y))^p
 q_t(x,y)\,\dd\mu(y)\,\dd\mu(x).
\end{equation*}

\begin{corollary}
\label{cor:transition-decoupling}
Assume that $a_t,b_t,a_\infty,b_\infty\in L^\infty_+(X,\mu)$ and
\[
 \sup_{t\ge1}\bigl(\|a_t\|_\infty+\|b_t\|_\infty\bigr)<\infty,
 \qquad
 a_t\stackrel{*}{\rightharpoonup}a_\infty,
 \qquad
 b_t\stackrel{*}{\rightharpoonup}b_\infty
 \quad\text{in }L^\infty(X,\mu),
\]
and, for every bounded $K\subset X$,
\begin{equation*}
 \int_K\int_Kq_t(x,y)\,\dd\mu(y)\,\dd\mu(x)
 \longrightarrow0.
\end{equation*}
Then for every $u\in L^p(X;Y)$,
\begin{equation*}
 J_p(t,u)\longrightarrow
 \int_X(a_\infty+b_\infty)d_Y(u,y_0)^p\,\dd\mu
 \qquad(t\to\infty).
\end{equation*}
\end{corollary}

\begin{proof}
Apply Theorem~\ref{thm:abstract-decoupling} to the measures \eqref{eq:Pi-qt}.
\end{proof}

The following semigroup-type Besov functional converts the large-time
behaviour of $J_p(t,u)$ into a small-order asymptotic problem.

\begin{definition}
For $s>0$ and $\beta>0$ define
\begin{equation}\label{eq:Besov-beta}
 \mathscr N_{s,p}^{q,\beta}(u)^p
 :=\int_0^\infty t^{-\beta sp-1}J_p(t,u)\,\dd t,
\end{equation}
whenever the right-hand side is finite.
\end{definition}

\begin{theorem}
\label{thm:Abel-semigroup}
Assume the hypotheses of Corollary~\ref{cor:transition-decoupling}.  Let
$u\in L^p(X;Y)$ and suppose that
\begin{equation}\label{eq:short-time-minimal-beta}
 \lim_{s\downarrow0}
 s\int_0^1t^{-\beta sp-1}J_p(t,u)\,\dd t=0.
\end{equation}
Then
\begin{equation}\label{eq:Abel-limit}
 \lim_{s\downarrow0}s\mathscr N_{s,p}^{q,\beta}(u)^p
 =\frac1{\beta p}\int_X(a_\infty+b_\infty)
 d_Y(u,y_0)^p\,\dd\mu.
\end{equation}
Since only the limit as $s\downarrow0$ is considered, it is enough to
assume that the integral defining $\mathscr N^{q,\beta}_{s,p}(u)$ is finite for
all sufficiently small $s>0$.
\end{theorem}

\begin{proof}
Split the integral at $t=1$.  The contribution from $(0,1)$ tends to zero
by \eqref{eq:short-time-minimal-beta}.  On $(1,\infty)$ define
\begin{equation*}
 \dd\eta_s(t):=
 \beta sp\,t^{-\beta sp-1}\one_{(1,\infty)}(t)\,\dd t.
\end{equation*}
Then $\eta_s$ is a probability measure because
\[
 \beta sp\int_1^\infty t^{-\beta sp-1}\,\dd t=1.
\]
For every $T>1$,
\[
 \eta_s((1,T))
 =1-T^{-\beta sp}\longrightarrow0,
\]
so $\eta_s$ escapes to $+\infty$.  By
Corollary~\ref{cor:transition-decoupling},
\[
 J_p(t,u)\longrightarrow
 L:=\int_X(a_\infty+b_\infty)d_Y(u,y_0)^p\,\dd\mu.
\]
The marginal bound and \eqref{eq:J-Pi-bound} give
\[
 0\le J_p(t,u)
 \le2^{p-1}\bigl(\|a_t\|_\infty+\|b_t\|_\infty\bigr)\Mp(u),
\]
so $J_p(t,u)$ is bounded for $t\ge1$. Splitting at a fixed $T>1$, using
uniform convergence to $L$ on $(T,\infty)$ and
$\eta_s((1,T))\to0$, yields
\[
 \int_1^\infty J_p(t,u)\,\dd\eta_s(t)\longrightarrow L.
\]
Finally,
\begin{align*}
 s\int_1^\infty t^{-\beta sp-1}J_p(t,u)\,\dd t
 &=\frac1{\beta p}
 \int_1^\infty J_p(t,u)\,\dd\eta_s(t)\\
 &\longrightarrow\frac{L}{\beta p}.
\end{align*}
Adding the short-time contribution proves \eqref{eq:Abel-limit}.
\end{proof}

\begin{proposition}
\label{prop:Besov-positive-order}
If
\begin{equation*}
 \int_0^1t^{-\beta s_0p-1}J_p(t,u)\,\dd t<\infty
\end{equation*}
for some $s_0>0$, then \eqref{eq:short-time-minimal-beta} holds.  Moreover,
\[
 \int_1^\infty t^{-\beta sp-1}J_p(t,u)\,\dd t
\]
is finite for every $s>0$.
\end{proposition}

\begin{proof}
For $0<s<s_0$ and $0<t<1$,
$t^{-\beta sp-1}\le t^{-\beta s_0p-1}$.  Hence
\[
 0\le
 s\int_0^1t^{-\beta sp-1}J_p(t,u)\,\dd t
 \le
 s\int_0^1t^{-\beta s_0p-1}J_p(t,u)\,\dd t
 \longrightarrow0.
\]
For $t\ge1$, the uniform $L^\infty$ bound on the marginal densities gives
$J_p(t,u)\le C\Mp(u)$, and
\[
 \int_1^\infty t^{-\beta sp-1}\,\dd t=\frac1{\beta sp}<\infty.
\]
\end{proof}

\subsection{Applications of the Abel principle}

We apply Theorem~\ref{thm:Abel-semigroup} to the non-symmetric
hypoelliptic semigroups of Buseghin--Garofalo--Tralli
\cite{BuseghinGarofaloTralli2022}, Riemannian heat semigroups, weighted
graphs, and symmetric conservative ultracontractive semigroups, including
the Baouendi--Grushin heat flow.

\subsubsection{Buseghin--Garofalo--Tralli and the two marginal masses}

Theorem~\ref{thm:Abel-semigroup} recovers the result of Buseghin, Garofalo,
and Tralli \cite{BuseghinGarofaloTralli2022}, including the change of
constant caused by the drift.

Let $Q=Q^*\ge0$ and $B\in\mathbb R^{N\times N}$, where $B^*$ denotes
transpose, and let
\begin{equation*}
 \mathscr A f(X)
 =\operatorname{tr}(Q\nabla^2f(X))
 +\langle BX,\nabla f(X)\rangle.
\end{equation*}
This is the H\"ormander-type operator considered in
\cite{BuseghinGarofaloTralli2022}.  We assume the
hypoellipticity condition used in \cite{BuseghinGarofaloTralli2022}, namely
\[
 K(t):=\frac1t\int_0^t e^{sB}Qe^{sB^*}\,\dd s>0
 \qquad\text{for every }t>0,
\]
together with the assumption
\begin{equation}\label{eq:trace-nonnegative}
 \operatorname{tr}B\ge0.
\end{equation}
Let $p(X,Z,t)$ be the heat kernel of the corresponding semigroup
$P_t^{\mathscr A}$.  Two identities recorded in
\cite[(2.1)--(2.2)]{BuseghinGarofaloTralli2022} are
\begin{equation}\label{eq:BGT-two-marginals}
 \int_{\mathbb R^N}p(X,Z,t)\,\dd Z=1,
 \qquad
 \int_{\mathbb R^N}p(X,Z,t)\,\dd X
 =e^{-t\operatorname{tr}B}.
\end{equation}
Thus
\begin{equation*}
 a_t=1,
 \qquad
 b_t=e^{-t\operatorname{tr}B}.
\end{equation*}

The kernel has the Gaussian form
\begin{equation*}
 p(X,Z,t)
 =\frac{c_N}{V(t)}
 \exp\left(-\frac{m_t(X,Z)^2}{4t}\right),
\end{equation*}
where $V(t)$ is the volume function introduced in that paper.  In
particular,
\begin{equation*}
 0\le p(X,Z,t)\le\frac{c_N}{V(t)}.
\end{equation*}
The large-time volume estimate proved in
\cite[Proposition~2.3]{BuseghinGarofaloTralli2022} implies
\begin{equation*}
 V(t)\longrightarrow\infty.
\end{equation*}
Therefore, if $K\subset\mathbb R^N$ is bounded,
\begin{equation*}
 \int_K\int_Kp(X,Z,t)\,\dd Z\,\dd X
 \le\frac{c_N|K|^2}{V(t)}\longrightarrow0.
\end{equation*}
Thus the transition couplings escape from bounded rectangles.

For a pointed metric target $(Y,d_Y,y_0)$ and
$u\in L^p(\mathbb R^N;Y)$ define
\begin{equation*}
 J_p^{\mathscr A}(t,u)
 :=\int_{\mathbb R^N}\int_{\mathbb R^N}
 d_Y(u(X),u(Z))^p p(X,Z,t)\,\dd Z\,\dd X.
\end{equation*}
By Corollary~\ref{cor:transition-decoupling},
\begin{equation*}
 J_p^{\mathscr A}(t,u)
 \longrightarrow
 \begin{cases}
 2\Mp(u),&\operatorname{tr}B=0,\\[1mm]
 \Mp(u),&\operatorname{tr}B>0.
 \end{cases}
\end{equation*}
Indeed, if $\operatorname{tr}B=0$ then
$a_\infty=b_\infty=1$, while if $\operatorname{tr}B>0$ then
$a_\infty=1$ and $b_\infty=0$.

Define the metric-valued semigroup Besov functional
\begin{equation*}
 \mathscr N_{s,p}^{\mathscr A}(u)^p
 :=\int_0^\infty
 t^{-sp/2-1}J_p^{\mathscr A}(t,u)\,\dd t.
\end{equation*}
This is \eqref{eq:Besov-beta} with $\beta=1/2$.

\begin{corollary}
\label{cor:BGT-metric}
Let $1\le p<\infty$, assume the preceding hypoellipticity condition, and assume \eqref{eq:trace-nonnegative}.  Let
$u\in L^p(\mathbb R^N;Y)$ and suppose that
\begin{equation}\label{eq:BGT-one-order}
 \mathscr N_{s_0,p}^{\mathscr A}(u)<\infty
\end{equation}
for some $s_0\in(0,1)$.  Then
\begin{equation}\label{eq:BGT-metric-limit}
 \lim_{s\downarrow0}s\mathscr N_{s,p}^{\mathscr A}(u)^p
 =
 \begin{cases}
 \dfrac4p\Mp(u),&\operatorname{tr}B=0,\\[2mm]
 \dfrac2p\Mp(u),&\operatorname{tr}B>0.
 \end{cases}
\end{equation}
For $Y=\mathbb R$, $y_0=0$, this recovers
\cite[Theorem~1.1]{BuseghinGarofaloTralli2022}.
\end{corollary}

\begin{proof}
By \eqref{eq:BGT-one-order}, the integral over $(0,1)$ at order $s_0$ is
finite.  Hence Proposition~\ref{prop:Besov-positive-order} gives the short-time
condition \eqref{eq:short-time-minimal-beta}.  The far-time integral is
finite for every $s>0$ because, by \eqref{eq:BGT-two-marginals},
\[
 J_p^{\mathscr A}(t,u)
 \le2^{p-1}\bigl(1+e^{-t\operatorname{tr}B}\bigr)\Mp(u)
 \le2^p\Mp(u).
\]
Apply Theorem~\ref{thm:Abel-semigroup} with $\beta=1/2$.  If
$\operatorname{tr}B=0$, then $a_\infty+b_\infty=2$, and
\[
 \frac{a_\infty+b_\infty}{\beta p}=
 \frac{2}{p/2}=\frac4p.
\]
If $\operatorname{tr}B>0$, then $a_\infty+b_\infty=1$, and the constant is
$2/p$.  This proves \eqref{eq:BGT-metric-limit}.
\end{proof}

\begin{remark}
For $\beta=1/2$, the Abelian factor is $2/p$. If
$\operatorname{tr}B=0$, both limiting marginal densities equal one; if
$\operatorname{tr}B>0$, the second tends to zero. This gives $4/p$ and
$2/p$, respectively.
\end{remark}
\subsection{Riemannian heat semigroup}
\begin{corollary}
\label{cor:riemannian-infinite-volume}
Let $(M,g)$ be a complete, connected Riemannian manifold with
\[
\mu(M)=+\infty,
\]
and let $h_t(x,y)$ be its minimal heat kernel. Set
\[
m_t(x):=\int_M h_t(x,y)\,\dd\mu(y)
\]
and
\[
\theta_M(x):=\lim_{t\to\infty}m_t(x).
\]
The latter limit exists for every $x\in M$ and defines a bounded
harmonic function with values in $[0,1]$.

Let $(Y,d_Y,y_0)$ be a pointed metric space, let $1\le p<\infty$,
and let $u\in L^p(M;Y)$. Define
\[
J_p(t,u)
:=
\int_M\int_M
d_Y(u(x),u(y))^p h_t(x,y)
\,\dd\mu(y)\,\dd\mu(x)
\]
and
\[
\mathscr N_{s,p}^{h,1/2}(u)^p
:=
\int_0^\infty
t^{-sp/2-1}J_p(t,u)\,\dd t.
\]
Assume that
\[
\mathscr N_{s_0,p}^{h,1/2}(u)<\infty
\]
for some $s_0>0$. Then
\begin{equation}\label{eq:riemannian-weighted-MS}
\lim_{s\downarrow0}
s\,\mathscr N_{s,p}^{h,1/2}(u)^p
=
\frac4p
\int_M
\theta_M(x)d_Y(u(x),y_0)^p\,\dd\mu(x).
\end{equation}
\end{corollary}

\begin{proof}
For $t>0$ define
\[
\dd\Pi_t(x,y)
:=
h_t(x,y)\,\dd\mu(x)\,\dd\mu(y).
\]
The symmetry and sub-Markov property of the heat kernel imply that both
marginals of $\Pi_t$ are $m_t\mu$ and $0\le m_t\le1$. Moreover,
$m_t(x)\to\theta_M(x)$ for every $x\in M$; see
\cite[Theorem~1.1(ii)]{CaselliGennaioli}. Dominated convergence therefore
gives
\[
 m_t\stackrel{*}{\rightharpoonup}\theta_M
 \quad\text{in }L^\infty(M,\mu).
\]

Let $P_t=e^{-t\Delta}$ be the minimal heat semigroup on $L^2(M,\mu)$.
For every bounded Borel set $K\subset M$, completeness and Hopf--Rinow
imply $\mu(K)<\infty$, and
\[
 \Pi_t(K\times K)=\langle\one_K,P_t\one_K\rangle_{L^2(M,\mu)}.
\]
By the spectral theorem, $P_t$ converges strongly as $t\to\infty$ to the
orthogonal projection onto $\ker\Delta$. Every function in this kernel has
zero Dirichlet energy and is therefore constant on the connected manifold.
Since $\mu(M)=\infty$, the only such $L^2$ function is zero. Hence
\[
 \Pi_t(K\times K)\longrightarrow0.
\]
Corollary~\ref{cor:transition-decoupling} now gives
\begin{equation}\label{eq:riemannian-large-time-weighted}
 J_p(t,u)\longrightarrow
 2\int_M\theta_M(x)d_Y(u(x),y_0)^p\,\dd\mu(x).
\end{equation}

The assumption at order $s_0$ and
Proposition~\ref{prop:Besov-positive-order} give the short-time condition
\eqref{eq:short-time-minimal-beta}. Applying
Theorem~\ref{thm:Abel-semigroup} with $\beta=1/2$ and limiting marginal
densities $a_\infty=b_\infty=\theta_M$ proves
\eqref{eq:riemannian-weighted-MS}.
\end{proof}
\begin{remark}
The result in \cite{CaselliGennaioli} is formulated for scalar-valued
functions and $p=2$. Corollary~\ref{cor:riemannian-infinite-volume} holds
for $1\le p<\infty$ and metric-valued maps.
\end{remark}

\subsubsection{Weighted graphs}

For weighted graphs, large-time escape follows from the spectral theorem
without imposing
volume doubling, a Poincar\'e inequality, or pointwise heat-kernel bounds.  We
refer to \cite{KellerLenzWojciechowski2021} for the analytic and probabilistic
background on graph Laplacians and heat semigroups.

Let $G=(V,b,m)$ be a connected, locally finite weighted graph, endowed with the
combinatorial distance, and let $L$ be the nonnegative self-adjoint
Friedrichs realization associated with the closed canonical Dirichlet form. Write
$P_t=e^{-tL}$ and denote by $p_t(x,y)$ its heat kernel with respect to $m$.
Thus
\[
 P_tf(x)=\sum_{y\in V}p_t(x,y)f(y)m(y).
\]
We say that $G$ is stochastically complete if
\[
 \sum_{y\in V}p_t(x,y)m(y)=1
 \qquad (x\in V,\ t>0).
\]

\begin{corollary}
\label{cor:graph-MS}
Assume that $G$ is connected, locally finite and stochastically complete, and
that $m(V)=+\infty$.  Let $(Y,d_Y,y_0)$ be a pointed metric space,
$1\le p<\infty$, and $\beta>0$.  For $u\in \ell^p(V,m;Y)$ set
\[
 J_p^G(t,u)
 :=\sum_{x,y\in V}d_Y(u(x),u(y))^p
 p_t(x,y)m(x)m(y)
\]
and
\[
 \mathscr N_{s,p}^{G,\beta}(u)^p
 :=\int_0^\infty t^{-\beta sp-1}J_p^G(t,u)\,\dd t.
\]
If $\mathscr N_{s_0,p}^{G,\beta}(u)<\infty$ for some $s_0>0$, then
\begin{equation}\label{eq:graph-MS}
 \lim_{s\downarrow0}s\,\mathscr N_{s,p}^{G,\beta}(u)^p
 =\frac{2}{\beta p}
 \sum_{x\in V}d_Y(u(x),y_0)^p m(x).
\end{equation}
\end{corollary}

\begin{proof}
Symmetry of the graph Dirichlet form and stochastic completeness imply that
both marginals of
\[
 \Pi_t(x,y):=p_t(x,y)m(x)m(y)
\]
are equal to $m$.  It remains to verify escape from bounded rectangles.  Since
$G$ is locally finite, bounded sets for the combinatorial distance are finite.
For such a set $K\subset V$,
\[
 \Pi_t(K\times K)
 =\langle \one_K,P_t\one_K\rangle_{\ell^2(V,m)}.
\]
By the spectral theorem, $P_t$ converges strongly in $\ell^2(V,m)$, as
$t\to\infty$, to the orthogonal projection onto $\ker L$.  If $f\in\ker L$,
then its Dirichlet energy vanishes, hence $f$ is constant on $V$ by
connectedness.  Since $m(V)=+\infty$, no nonzero constant belongs to
$\ell^2(V,m)$, and therefore $\ker L=\{0\}$.  It follows that
$\Pi_t(K\times K)\to0$.  Corollary~\ref{cor:transition-decoupling} gives
\[
 J_p^G(t,u)\longrightarrow
 2\sum_{x\in V}d_Y(u(x),y_0)^p m(x),
\]
and Theorem~\ref{thm:Abel-semigroup}, together with
Proposition~\ref{prop:Besov-positive-order}, yields \eqref{eq:graph-MS}.
\end{proof}

\begin{remark}
For a finite connected weighted graph, the heat semigroup converges instead to
the projection onto constants.  If $M=m(V)$, the same Abelian argument gives
\[
 \lim_{s\downarrow0}s\,\mathscr N_{s,p}^{G,\beta}(u)^p
 =\frac{1}{\beta p M}
 \sum_{x,y\in V}d_Y(u(x),u(y))^p m(x)m(y).
\]
Moreover, on
graphs with anomalous diffusion and walk dimension $d_w$, the spatially
natural choice is $\beta=1/d_w$; formula \eqref{eq:graph-MS} then becomes
\[
 \lim_{s\downarrow0}s\,\mathscr N_{s,p}^{G,1/d_w}(u)^p
 =\frac{2d_w}{p}\sum_{x\in V}d_Y(u(x),y_0)^p m(x).
\]
This specialization applies whenever a walk dimension $d_w$ is specified by
two-sided sub-Gaussian heat-kernel estimates; see, e.g.,
\cite{GrigoryanTelcs2001}.
\end{remark}

\subsubsection{Symmetric conservative ultracontractive semigroups and the Baouendi--Grushin example}

The following result applies to the Baouendi--Grushin heat semigroup.

\begin{corollary}
\label{cor:symmetric-ultracontractive-MS}
Assume that the kernels $q_t$, $t>0$, are symmetric,
\[
 q_t(x,y)=q_t(y,x),
\]
and conservative,
\[
 \int_Xq_t(x,y)\,\dd\mu(y)=1
 \qquad\text{for a.e. }x\in X.
\]
Let
\[
 \mathsf P_tf(x)
 :=
 \int_Xq_t(x,y)f(y)\,\dd\mu(y),
\]
and suppose that there exist $C,\gamma>0$ such that
\begin{equation*}
 \|\mathsf P_tf\|_{L^\infty(X)}
 \le Ct^{-\gamma}\|f\|_{L^1(X)}
 \qquad(t\ge1).
\end{equation*}
Then the two marginals of
\[
 \dd\Pi_t(x,y)
 =
 q_t(x,y)\,\dd\mu(x)\,\dd\mu(y)
\]
are both equal to $\mu$, and $\Pi_t$ escapes from bounded rectangles.

Consequently, for every pointed metric space $(Y,d_Y,y_0)$,
$1\le p<\infty$, and $u\in L^p(X;Y)$ such that
\[
 \mathscr N_{s_0,p}^{q,1/2}(u)<\infty
\]
for some $s_0>0$, one has
\begin{equation*}
 \lim_{s\downarrow0}
 s\,\mathscr N_{s,p}^{q,1/2}(u)^p
 =
 \frac4p\,\Mp(u).
\end{equation*}
\end{corollary}

\begin{proof}
Conservativeness gives the first marginal of $\Pi_t$, while symmetry gives
the second one.  Thus $a_t=b_t=1$.  If $K\subset X$ is bounded, then
\[
 \Pi_t(K\times K)
 =
 \int_K\mathsf P_t\one_K(x)\,\dd\mu(x)
 \le
 Ct^{-\gamma}\mu(K)^2
 \longrightarrow0.
\]
Hence Corollary~\ref{cor:transition-decoupling} applies and gives
\[
 J_p(t,u)\longrightarrow2\Mp(u)
 \qquad(t\to\infty).
\]
The short-time assumption in Theorem~\ref{thm:Abel-semigroup} follows
from Proposition~\ref{prop:Besov-positive-order}.  Taking
$\beta=1/2$ in \eqref{eq:Abel-limit} yields
\[
 \frac{a_\infty+b_\infty}{\beta p}
 =
 \frac{2}{p/2}
 =
 \frac4p.
\]
\end{proof}

We apply Corollary~\ref{cor:symmetric-ultracontractive-MS} to the heat kernel
of the generally non-hypoelliptic Baouendi--Grushin operator. This recovers
the asymptotic formula in \cite{LiWang2024}; see
\cite[Section~2]{LiWang2024} for background.

\begin{corollary}
\label{cor:Grushin-MS}
Let $\alpha\ge0$, let $k,m\ge1$, $n=k+m$, and consider
\[
 G_\alpha^n=\mathbb R^k\times\mathbb R^m
\]
with the Baouendi--Grushin operator
\[
 \Delta_\alpha
 =
 \Delta_x+|x|^{2\alpha}\Delta_y.
\]
Let $(H_t)_{t\ge0}$ be its heat semigroup and let
$h_t(g,g')$ be the corresponding heat kernel with respect to Lebesgue
measure.  Let $(Y,d_Y,y_0)$ be a pointed metric space and
$1\le p<\infty$.  For
$u\in L^p(G_\alpha^n;Y)$ define
\begin{equation*}
 J_p^\alpha(t,u)
 :=
 \int_{G_\alpha^n}\int_{G_\alpha^n}
 d_Y(u(g),u(g'))^p
 h_t(g,g')
 \,\dd g'\,\dd g
\end{equation*}
and
\begin{equation*}
 \mathscr N_{s,p}^\alpha(u)^p
 :=
 \int_0^\infty
 t^{-sp/2-1}J_p^\alpha(t,u)\,\dd t.
\end{equation*}
If
\[
 \mathscr N_{s_0,p}^\alpha(u)<\infty
\]
for some $s_0>0$, then
\begin{equation*}
 \lim_{s\downarrow0}
 s\,\mathscr N_{s,p}^\alpha(u)^p
 =
 \frac4p
 \int_{G_\alpha^n}d_Y(u(g),y_0)^p\,\dd g.
\end{equation*}
For $Y=\mathbb R$, this is the 
Maz'ya--Shaposhnikova formula of Li and Wang
\cite[Theorem~2.3]{LiWang2024}.
\end{corollary}

\begin{proof}
The Grushin heat semigroup is conservative and its heat kernel is symmetric;
see \cite[Section~2]{LiWang2024}.  If
\[
 Q=n+\alpha m
\]
denotes the homogeneous dimension, the ultracontractive estimate
\cite[Lemma~2.1]{LiWang2024} gives
\begin{equation*}
 \|H_tf\|_{L^\infty(G_\alpha^n)}
 \le Ct^{-Q/2}\|f\|_{L^1(G_\alpha^n)}.
\end{equation*}
Consequently, for every bounded $K\subset G_\alpha^n$,
\[
 \int_K\int_Kh_t(g,g')\,\dd g'\,\dd g
 =
 \int_KH_t\one_K(g)\,\dd g
 \le
 Ct^{-Q/2}|K|^2
 \longrightarrow0.
\]
Thus the heat-kernel couplings have both marginals equal to Lebesgue
measure and escape from bounded rectangles.  The conclusion follows from
Corollary~\ref{cor:symmetric-ultracontractive-MS}.
\end{proof}
\section{Application to Carnot groups}
\label{sec:Carnot}

In this section we apply the abstract decoupling theorem to metric-valued maps on Carnot groups.

\medskip
\noindent Let $G$ be a Carnot group with dilations $(\delta_r)_{r>0}$, homogeneous
dimension $Q$, Haar measure $\dd x$, and a continuous proper homogeneous norm
$N_G$. For some basic facts on fractional Sobolev spaces on Carnot groups we refer,
for instance, to \cite{CapolliMaioneSalortVecchi2021}. Set
\[
 B_G(e,r):=\{h\in G:N_G(h)<r\},
 \qquad
 S_G:=\{h\in G:N_G(h)=1\}.
\]
We use the polar decomposition
\[
 \int_G f(h)\,\dd h
 =
 \int_0^\infty\int_{S_G}
 f(\delta_r\omega)\,
 \dd\sigma_G(\omega)\,r^{Q-1}\,\dd r,
\]
where $\sigma_G$ is the corresponding polar measure on $S_G$. In particular,
\[
 \sigma_G(S_G)=Q|B_G(e,1)|.
\]
Let $(Y,d_Y,y_0)$ be a pointed metric space and let $1\le p<\infty$. For
$u\in L^p(G;Y)$ define
\[
 \mathcal E_{s,p}(u)
 :=
 \int_G\int_G
 \frac{d_Y(u(x),u(y))^p}
      {N_G(x^{-1}y)^{Q+sp}}
 \,\dd y\,\dd x
\]
and
\[
 F_u(h):=\int_G d_Y(u(x),u(xh))^p\,\dd x.
\]
By the change of variables $y=xh$ and unimodularity,
\begin{equation}\label{eq:Carnot-energy-F}
 \mathcal E_{s,p}(u)
 =
 \int_G
 \frac{F_u(h)}
      {N_G(h)^{Q+sp}}
 \,\dd h.
\end{equation}

The lemma is a direct consequence of
Theorem~\ref{thm:abstract-decoupling}.

\begin{lemma}
\label{lem:Carnot-translation-decoupling}
Let $(\nu_\alpha)_\alpha$ be Borel probability measures on $G$ such that
$\nu_\alpha(K)\to0$ for every compact set $K\subset G$. Then, for every
$u\in L^p(G;Y)$,
\[
 \int_G F_u(h)\,\dd\nu_\alpha(h)
 \longrightarrow 2\Mp(u).
\]
\end{lemma}

\begin{proof}
Associate with $\nu_\alpha$ the coupling $\Pi_\alpha$ defined by
\[
 \int_{G\times G}\Phi(x,y)\,\dd\Pi_\alpha(x,y)
 :=
 \int_G\int_G \Phi(x,xh)\,\dd x\,\dd\nu_\alpha(h)
\]
for every nonnegative Borel function $\Phi$. Since every Carnot group is
unimodular, both marginals of $\Pi_\alpha$ are Haar measure. Moreover, if
$K\subset G$ is bounded, then $\overline K^{-1}\overline K$ is compact and
\[
 \Pi_\alpha(K\times K)
 \le |K|\,\nu_\alpha(\overline K^{-1}\overline K)
 \longrightarrow0.
\]
Thus $\Pi_\alpha$ escapes from bounded rectangles. If $Y$ has at least two
points, Theorem~\ref{thm:abstract-decoupling} applies with
$a_\alpha=b_\alpha=1$; the one-point case is trivial. Since
\[
 J_{\Pi_\alpha}(u)=\int_G F_u(h)\,\dd\nu_\alpha(h),
\]
the conclusion follows.
\end{proof}

\begin{theorem}
\label{thm:Carnot-MS}
Let $1\le p<\infty$ and $u\in L^p(G;Y)$. Assume that, for all sufficiently
small $s>0$,
\[
 \int_{\{N_G(h)<1\}}
 \frac{F_u(h)}
      {N_G(h)^{Q+sp}}
 \,\dd h<\infty
\]
and that
\begin{equation}\label{eq:Carnot-minimal-short}
 \lim_{s\downarrow0}
 s\int_{\{N_G(h)<1\}}
 \frac{F_u(h)}
      {N_G(h)^{Q+sp}}
 \,\dd h
 =0.
\end{equation}
Then
\begin{equation}\label{eq:Carnot-MS}
 \lim_{s\downarrow0}
 s\mathcal E_{s,p}(u)
 =
 \frac{2\sigma_G(S_G)}{p}\Mp(u).
\end{equation}
In particular, \eqref{eq:Carnot-MS} holds whenever
\begin{equation}\label{eq:Carnot-positive-order}
 \mathcal E_{\theta,p}(u)<\infty
\end{equation}
for some $\theta>0$.
\end{theorem}

\begin{proof}
For $s>0$ define
\begin{equation}\label{eq:nu-s-Carnot}
 \dd\nu_s(h)
 :=
 \frac{sp}{\sigma_G(S_G)}
 \frac{\one_{\{N_G(h)>1\}}}
      {N_G(h)^{Q+sp}}
 \,\dd h.
\end{equation}
By the polar formula,
\[
 \int_G\dd\nu_s
 =
 \frac{sp}{\sigma_G(S_G)}\sigma_G(S_G)
 \int_1^\infty r^{-1-sp}\,\dd r
 =1.
\]
Thus $\nu_s$ is a probability measure. Moreover, for every $R>1$,
\[
 \nu_s(\{1<N_G\le R\})
 =
 sp\int_1^Rr^{-1-sp}\,\dd r
 =
 1-R^{-sp}
 \longrightarrow0
\]
as $s\downarrow0$. Since every compact subset of $G$ is contained in some
ball, $(\nu_s)_s$ escapes to infinity.

Split \eqref{eq:Carnot-energy-F} into the regions $N_G(h)<1$ and
$N_G(h)>1$. By \eqref{eq:Carnot-minimal-short}, the normalized local
contribution vanishes. Using \eqref{eq:nu-s-Carnot} we can write
\[
 s\int_{\{N_G(h)>1\}}
 \frac{F_u(h)}{N_G(h)^{Q+sp}}\,\dd h
 =
 \frac{\sigma_G(S_G)}{p}
 \int_G F_u(h)\,\dd\nu_s(h).
\]
Lemma~\ref{lem:Carnot-translation-decoupling} yields
\[
 \int_G F_u(h)\,\dd\nu_s(h)
 \longrightarrow2\Mp(u),
\]
and \eqref{eq:Carnot-MS} follows.
It remains to prove the last assertion. Assume that
\eqref{eq:Carnot-positive-order} holds for some $\theta>0$. For
$0<s<\theta$ and $N_G(h)<1$,
\[
 N_G(h)^{-Q-sp}
 \le
 N_G(h)^{-Q-\theta p}.
\]
Consequently,
\[
 0
 \le
 s\int_{\{N_G(h)<1\}}
 \frac{F_u(h)}{N_G(h)^{Q+sp}}\,\dd h
 \le
 s\mathcal E_{\theta,p}(u)
 \longrightarrow0.
\]
Thus \eqref{eq:Carnot-minimal-short} holds and the first part of the theorem
applies.
\end{proof}

The theorem also gives a corresponding statement for maps between
Carnot groups. Let $H$ be another Carnot group endowed with a left-invariant
homogeneous distance $d_H$, and write
\[
 N_H(z):=d_H(e_H,z).
\]

\begin{corollary}
\label{cor:Carnot-to-Carnot}
Let $u:G\to H$ satisfy
\[
 \int_G N_H(u(x))^p\,\dd x<\infty.
\]
Assume 
\[
 \lim_{s\downarrow0}
 s\int_{\{N_G(h)<1\}}
 \int_G
 \frac{N_H(u(x)^{-1}u(xh))^p}
      {N_G(h)^{Q+sp}}
 \,\dd x\,\dd h
 =0.
\]
Then
\begin{equation}\label{eq:Carnot-to-Carnot-MS}
 \begin{aligned}
 &\lim_{s\downarrow0}
 s\int_G\int_G
 \frac{N_H(u(x)^{-1}u(y))^p}
      {N_G(x^{-1}y)^{Q+sp}}
 \,\dd y\,\dd x
 \\
 &\qquad=
 \frac{2\sigma_G(S_G)}{p}
 \int_G N_H(u(x))^p\,\dd x.
 \end{aligned}
\end{equation}
\end{corollary}

\begin{proof}
Apply Theorem~\ref{thm:Carnot-MS} with $(Y,d_Y,y_0)=(H,d_H,e_H)$.
\end{proof}

\begin{corollary}
\label{cor:Carnot-heat-MS}
Let $p_t$ be the heat kernel of a homogeneous sub-Laplacian on $G$. For a
pointed metric space $(Y,d_Y,y_0)$, $1\le p<\infty$, and $u\in L^p(G;Y)$,
set
\begin{equation*}
 J_p^G(t,u)
 :=\int_G\int_G d_Y(u(g),u(h))^p
 p_t(g^{-1}h)\,\dd h\,\dd g
\end{equation*}
and
\begin{equation*}
 \mathscr N_{s,p}^{G,h}(u)^p
 :=\int_0^\infty t^{-sp/2-1}J_p^G(t,u)\,\dd t.
\end{equation*}
If $\mathscr N_{s_0,p}^{G,h}(u)<\infty$ for some $s_0>0$, then
\begin{equation}\label{eq:Carnot-heat-MS}
 \lim_{s\downarrow0}s\mathscr N_{s,p}^{G,h}(u)^p
 =\frac4p\int_Gd_Y(u(g),y_0)^p\,\dd g.
\end{equation}
\end{corollary}

\begin{proof}
Let us start recalling that
\begin{equation}\label{intw}
 \int_Gp_t(h)\,\dd h=1
\end{equation}
and 
\begin{equation}\label{scaling}
 p_t(h)
 =
 t^{-Q/2}p_1(\delta_{t^{-1/2}}h).
 \end{equation}
Define
\[
 \dd\Pi_t(g,h):=p_t(g^{-1}h)\,\dd g\,\dd h.
\]
The heat kernel is conservative and symmetric. Since the Haar measure on a Carnot
group is unimodular\footnote{A locally compact group $G$ is called unimodular if its left
and right Haar measures coincide. Equivalently, its Haar measure
is invariant under both left and right translations} and \eqref{intw}, we have that for every $g,h\in G$,
\[
 \int_Gp_t(g^{-1}h)\,\dd h=1,
 \qquad
 \int_Gp_t(g^{-1}h)\,\dd g=1.
\]
Thus both marginals of $\Pi_t$ are Haar measure.
The time-one heat kernel is bounded. The scaling property \eqref{scaling} gives
\[
 \|p_t\|_{L^\infty(G)}
 =t^{-Q/2}\|p_1\|_{L^\infty(G)}.
\]
Consequently, for every bounded Borel set $K\subset G$,
\[
 \Pi_t(K\times K)
 \le |K|^2\|p_t\|_{L^\infty(G)}
 =|K|^2t^{-Q/2}\|p_1\|_{L^\infty(G)}
 \longrightarrow0.
\]
The couplings therefore escape from bounded rectangles, and
Corollary~\ref{cor:transition-decoupling} yields
\[
 J_p^G(t,u)\longrightarrow2\Mp(u).
\]
The positive-order assumption implies the short-time condition by
Proposition~\ref{prop:Besov-positive-order}. Theorem~\ref{thm:Abel-semigroup}
with $\beta=1/2$ then gives \eqref{eq:Carnot-heat-MS}.
\end{proof}

\begin{remark}[Comparison with \cite{GarofaloTralli2024}]
\label{rem:GT-comparison}
For scalar-valued maps, Corollary~\ref{cor:Carnot-heat-MS} coincides with
\cite[Theorem~1.5]{GarofaloTralli2024} by Garofalo and Tralli.  With their notation,
\[
 \mathcal N_{s,p}^{\mathbb G}(f)^p
 :=\int_0^\infty t^{-sp/2-1}
 \int_{\mathbb G}P_t\bigl(|f-f(g)|^p\bigr)(g)\,\dd g\,\dd t,
\]
and for $1\le p<\infty$ they prove
\[
 \lim_{s\downarrow0}s\,\mathcal N_{s,p}^{\mathbb G}(f)^p
 =\frac4p\|f\|_{L^p(\mathbb G)}^p
\]
for $f$ belonging to one positive heat-Besov class.  A central ingredient of
their proof is an integral decoupling property of the Carnot heat kernel.
In the language of the present paper, that decoupling property says that the
measures
\[
 \dd\Pi_t(g,h)=p_t(g^{-1}h)\,\dd g\,\dd h
\]
escape every bounded rectangle as $t\to\infty$.  Conservation and
self-adjointness give two Haar marginals of mass one.  Hence
Theorem~\ref{thm:abstract-decoupling} yields
$J_p(t,f)\to2\|f\|_p^p$, and the Abel factor associated with
$t^{-sp/2-1}$ is $2/p$, producing the constant $4/p$.
\end{remark}

\section*{Use of generative AI}
OpenAI's ChatGPT SOL was used for language editing and exploratory
discussion. The author is responsible
for the contents of the  manuscript.

\end{document}